\documentclass{amsart}
\usepackage{amsfonts}
\usepackage{amsmath}
\usepackage{amssymb}
\usepackage{amsthm}
\usepackage{newlfont}

\newtheorem{lemma}{Lemma}

\newtheorem{theorem}{Theorem}[section]

\newtheorem{exm}{Example}
\allowdisplaybreaks

\begin{document}

\title{Asymptotic Behaviour of an Infinitely-Many-Alleles Diffusion with Symmetric Overdominance}

\author{Youzhou Zhou}
\address{Department of Mathematics and Statistics\\
McMaster University, 1280 Main Street West,\\
Hamilton Ontario, Canada L8S 2L4}
\email{zhouy52@math.mcmaster.ca}

\subjclass[2000]{Primary }


\date{\today}

\dedicatory{}

\begin{abstract}
This paper considers the limiting distribution of $\pi_{\lambda,\theta}$, the stationary distribution of the infinitely-many-alleles diffusion with symmetric overdominance
 \cite{MR1626158}.  In \cite{MR2519357} the large deviation principle for $\pi_{\lambda,\theta}$ indicates that there are countably many phase transitions for the limiting distribution of $\pi_{\lambda,\theta}$, and the critical points are $\lambda=k(k+1), k\geq1$. The asymptotic behaviours at those critical points , however, are unclear. This article provides a definite description of the critical cases.
\end{abstract}

\maketitle

\section{Introduction}

  The infinitely many alleles model is an extensively studied model in population genetics. In this model, mutations always generate completely new allele types, and $x=(x_{1},x_{2},\cdots)$, where $x_{i}$'s are arranged decreasingly and $\sum_{i=1}^{\infty}x_{i}=1$, is usually used to represent the allele frequency.   The infinitely-many-neutral-alleles diffusion \cite{MR615945} is the associated diffusion process characterized by generator
  $$
G=\frac{1}{2}\sum_{i,j=1}^{\infty}x_{i}(\delta_{ij}-x_{j})\frac{\partial^{2}}{\partial x_{i}\partial x_{j}}-\frac{\theta}{2}\sum_{i=1}^{\infty}x_{i}\frac{\partial}{\partial x_{i}}, x\in \bar{\triangledown}_{\infty},
$$
where $\bar{\triangledown}_{\infty}=\{(x_{1},x_{2},\cdots)|x_{1}\geq x_{2}\geq\cdots\geq0,\sum_{i=1}^{\infty}x_{i}\leq1\}.$
  The Poisson-Dirichlet distribution \cite{MR0368264}, hereafter denoted as PD($\theta$), is its stationary distribution. If symmetric overdominant selection is considered, we will end up with the infinitely-many-alleles diffusion with symmetric overdominance\\
  \cite{MR1626158}, characterized by 
  $$
G_{\sigma}=G+\sigma\sum_{i=1}^{\infty}x_{i}(x_{i}-\varphi_{2}(x))\frac{\partial}{\partial x_{i}}, x\in\bar{\triangledown}_{\infty}, 
$$
where $\varphi_{2}(x)=\sum_{i=1}^{\infty}x_{i}^{2}$. In population genetics, the homozygosity, denoted by H$_{2}$, is defined to be $\varphi_{2}(x)$ for a given allele frequency $x$.  H$_{2}$ is a random variable for allele frequency is random.

 The infinitely-many-alleles diffusion with symmetric overdominance has stationary distribution $\pi_{\sigma}$ defined as 
 $$
\pi_{\sigma}(dx)=C_{\sigma}\exp\{\sigma \varphi_{2}(x)\}\mbox{PD}(\theta)(dx), x\in\triangledown_{\infty},
$$
where $C_{\sigma}$ is a normalized constant.  If $\sigma>0$, then the selection is underdominant; if $\sigma<0$, then the selection is overdominant. In this article, only overdominant selection is considered.

For the infinitely-many-alleles diffusion with symmetric overdominance, random sampling, mutations and selection are all the evolutionary forces involved. It is commonly accepted that mutations contribute substantially to the genetic variability. Even when the very low mutation rate is presented,  overdominant selection, however, can maintain certain amount of biological diversity, please refer to \cite{MR64600} and references therein. Random sampling is the evolutionary force constantly deleting some types. The interactions of those forces determine the unique configuration of the whole population. When mutations and overdominant selection are both large, the effect of overdominant selection is hardly pronounced. This intuitive statement is in fact verified by the numerical results in \cite{Gillespie_role_of_size}. Also some numerical results in \cite{Gillespie_role_of_size} are confirmed  theoretically by Joyce, Krone and Kurtz \cite{MR1951997}, Dawson and Feng \cite{MR2244425}. On the contrary, when mutation is small and overdominant selection is large, the effect of overdominant selection can be observable.  Similarly, this statement is also shown in \cite{MR2519357} through large deviation principle (hereafter LDP, for detailed introduction of LDP, please refer to \cite{MR2571413}) for PD($\theta$)  and $\pi_{\sigma}$ with small mutation and large selection. For PD($\theta$), the speed is $\log\frac{1}{\theta}$ and the LDP rate function is 
$$
J(x)=
\begin{cases}
0,             & x\in L_{1}\\
n-1,           & x\in L_{n},x_{n}>0,n\geq2\\
+\infty,       &x\notin L
\end{cases}
$$
where $L_{n}=\{(x_{1},\cdots,x_{n},0,\cdots)\in\bar{\triangledown}_{\infty}\mid\sum_{i=1}^{n}x_{i}=1\}, 
L=\bigcup_{n=1}^{\infty}L_{n}$.
Here $J(x)$ exhibits some properties similar to energy ladder structure. In order to understand the interaction of small mutation and large selection,  the selection intensity $\sigma$ is regarded as $\sigma(\theta)$, and the LDP for $\pi_{\sigma(\theta)}$ was also considered.  Especially, when $\sigma=\lambda\log\theta(\lambda>0,0<\theta<1)$, $\pi_{\sigma(\theta)}$ is denoted as $\pi_{\lambda,\theta}$. The rate function of the LDP for $\pi_{\lambda,\theta}$ is 
$$
S_{\lambda}(x)=
J(x)+\lambda \varphi_{2}(x)-\inf\left\{\frac{\lambda}{n}+n-1\mid n\geq1\right\}.
$$
Thus, the effects of overdominant selection are pronounced because the LDP for PD($\theta$) and the LDP for $\pi_{\lambda,\theta}$ have different rate function. It was observed in \cite{MR2519357} that when $\lambda\in(k(k-1),k(k+1)),k\geq1$, the limiting distribution of $\pi_{\lambda,\theta}$ is $\delta_{(\frac{1}{k},\cdots,\frac{1}{k},0,\cdots)}$; but for the critical case $\lambda=k(k+1),k\geq1$, the LDP rate function $S(x)$ has two zero points. Therefore the law of large numbers of $\pi_{\lambda.\theta}$ at critical points remains open. 

 The main result of this paper confirms that the asymptotic distribution at critical values $\lambda=k(k+1)$ is also 
 $$
 \delta_{(\frac{1}{k},\cdots,\frac{1}{k},0,\cdots)}.
 $$
  In general, as $\theta\rightarrow0$, the limiting distribution of $\pi_{\lambda,\theta}$  can be written as  
 $$
 \sum_{k=1}^{\infty}I_{(k(k-1),k(k+1)]}(\lambda)\delta_{(\frac{1}{k},\cdots,\frac{1}{k},0,\cdots)}.
 $$
 Therefore, for different selection intensity, the asymptotic distribution varies, and there are countably many phase transitions. This is indeed a bit of surprising.
 
 The possible explanation of the phase transition can be two fold. On one hand, mathematically, $\varphi_{2}(x)=\sum_{i=1}^{\infty}x_{i}^{2}$ can be regarded as the potential function of the infinitely-many-alleles diffusion with symmetric overdominance. Obviously, $\varphi_{2}(x)$ has a minimum point $(\frac{1}{k},\cdots,\frac{1}{k},0,\cdots)$ in each domain $L_{k},k\geq1$. The graph of $\varphi_{2}(x)$ thus indicates a ``multi-valley energy landscape". For a given selection intensity, the system is trapped in a specific valley with bottom point say $(\frac{1}{k},\cdots,\frac{1}{k},0,\cdots).$ On the other hand, biologically, random sampling constantly deletes a great amount of gene types. Therefore, the system is often likely to stay in $L_{k},k\geq1.$ Once the mutation is present, it will move the system upward along the energy ladder, i.e. the system will gradually move to $L_{k+1}$ from $L_{k},k\geq1.$ The symmetry of the overdominant selection guarantees that the existed types are evenly distributed.  Hence, the three evolutionary forces will balance out a single state such as $(\frac{1}{k},\cdots,\frac{1}{k},0,\cdots)$.

Furthermore, the limiting distribution of homozygosity under $\pi_{\lambda,\theta}$ is also obtained and it is
\begin{equation*}
 \sum_{k=1}^{\infty}I_{(k(k-1),k(k+1)]}(\lambda)\delta_{\frac{1}{k}}. \label{Homozygosity}
 \end{equation*}
 
 All these results are concerned with the infinitely many alleles model. For finitely many alleles model, say two alleles model, the asymptotic distribution of homozygosity is quite similar to the infinitely many alleles model and reads as 
 $$
 I_{(0,2]}(\lambda)\delta_{\frac{1}{2}}+I_{(2,\infty)}(\lambda)\delta_{1}.
 $$
 But the limiting distribution of the stationary distribution for the two alleles model with symmetric overdominance is different since it is a labelled model. Please refer to \cite{MThesis} for the derivation of the limiting distribution for the two alleles model with symmetric overdominance. Presumably, for other finitely many alleles model, the asymptotic homozygosity should behave similarly; but the proof is missing.
  
The whole paper is organized as follows. In section 2,  we will present the main theorem on limiting distribution of homozygosity and its proof. In section 3, proofs of all lemmas will be shown in detail.

\section{Main results}

In this section, we are going to derive the limiting distribution of $\pi_{\lambda,\theta}$ at the critical points $\lambda=k(k+1),k\geq1.$ Due to the LDP estimation of $\pi_{\lambda,\theta}$, we can conclude that the limiting distribution of $\pi_{\lambda,\theta}$, if any, should concentrate on two points $(\frac{1}{k},\cdots,\frac{1}{k},0,\cdots)$ and $(\frac{1}{k+1},\cdots,\frac{1}{k+},0,\cdots)$. But we can not determine the probability weights of these two points using only LDP estimation of $\pi_{\lambda,\theta}$.   

 To this end, we are going to find the limits of homozygosity H$_{2}$ first. Then, making use of the LDP estimation of $\pi_{\lambda,\theta}$, the limits of $\pi_{\lambda,\theta}$ are finally obtained. To obtain the limit of H$_{2}$, we need to estimate its asymptotic moment generating function.  Hence, we expand terms such as the normalized constant $C_{\sigma}$, which is usually called partition function in statistical physics. Since $C_{\sigma}$ is a function of $\theta$, we will expand it near $0$. Thanks to a ratio limit theorem, i.e. Lemma \ref{ML}, the limiting moment generating function of H$_{2}$ will be determined by the leading term in $C_{\sigma}$'s expansion.

\begin{theorem}\label{Homozygosity_LLN}
For $\lambda>0$, the limiting distributions of homozygosity, H$_{2}$, under the distribution $\pi_{\theta,\lambda}$ is 
$$
\sum_{k=1}^{+\infty}{I}_{(k(k-1),k(k+1)]}(\lambda)\delta_{\frac{1}{k}}(dx).
$$
\end{theorem}

 Before we present the proof, we need the following lemmas, the proofs of which are postponed.
 
 \begin{lemma}\label{MH}
 The moment of the heterozygosity $m_{k}=E(1-H_{2})^{k}$ has the form
 $$
 m_{k}=\sum_{l=1}^{k}A_{k,l}(\theta)\theta^{l},
 $$
 where \begin{align*}
 A_{k,1}(\theta)&=\frac{2k+\theta}{2k}\frac{2^{k}k!\Gamma(k+\theta)}{\Gamma(2k+1+\theta)}\\
 A_{k,p}(\theta)&=\sum_{l=p-1}^{k-1}\frac{2k+\theta}{2k}\frac{2^{k}k!}{2^{l}l!}\frac{\Gamma(k+l+\theta)}{\Gamma(2k+1+\theta)}A_{l,p-1}(\theta), \quad p\geq 2.
 \end{align*}  
 \end{lemma}
 
 Let us define $A_{k,p}=A_{k,p}(0)$; then 
 $$
 A_{k,1}=\frac{2^{k}k!(k-1)!}{(2k)!},
 $$
 and
 $$
 A_{k,p}=\sum_{l=p-1}^{k-1}\frac{2^{k}k!\Gamma(k+l)}{2^{l}l!\Gamma(2k+1)}A_{l,p-1},  \quad p\geq2.
 $$
 Thus $A_{k,p}$ does not depend on $\theta$ anymore; but it is an appropriate approximation of $A_{k,p}(\theta)$, as can be seen in the proof of Lemma \ref{COEF}.  
  \begin{lemma}\label{UBMH}
 If we fix integer $p\geq1$, then, $\forall \theta\in[0,1]$, we have 
\begin{align*}
& \frac{1}{2^{p}}A_{k,p}\leq A_{k,p}(\theta)\leq A_{k,p}\leq \frac{1}{2^{p-2}}, &\forall k\geq p\geq1;\\
&|A_{k,p}(\theta)-A_{k,p}|\leq \theta p A_{k,p}, &1\leq p\leq k.
 \end{align*}
 \end{lemma}
 
\begin{lemma} \label{FT}
For $\lambda>0$, we have
 $$
 \lim_{\theta\rightarrow0}\sum_{l=[\lambda]+1}^{\infty}\theta^{l}\sum_{k=l}^{\infty}\frac{(\lambda\log\frac{1}{\theta})^{k}}{k!}A_{k,l}(\theta)=0.
 $$
 and
 $$
  \lim_{\theta\rightarrow0}\sum_{l=[\lambda]+1}^{\infty}\theta^{l}\sum_{k=l}^{\infty}\frac{(\lambda\log\frac{1}{\theta})^{k}}{k!}A_{k+n,l}(\theta)=0.
 $$
\end{lemma} 

 \begin{lemma}\label{ML}
 Let $a_{n}, b_{n}$ be two positive sequences. Suppose that $\lim_{n\rightarrow\infty}\frac{a_{n}}{b_{n}}=c$ and $\sum_{n=0}^{\infty}a_{n}\frac{x^{n}}{n!}$ and $\sum_{n=0}^{\infty}b_{n}\frac{x^{n}}{n!}$ are both convergent in $\mathbb{R}$. Then $\lim_{x\rightarrow+\infty}\frac{\sum_{n=0}^{\infty}a_{n}\frac{x^{n}}{n!}}{\sum_{n=0}^{\infty}b_{n}\frac{x^{n}}{n!}}=c.$
 \end{lemma}
 
 \begin{lemma}\label{COEF1}
 For any fixed integer $p\geq1$,  we have, as $k\rightarrow+\infty$,
  $$
 A_{k,p}\sim C_{p}\frac{1}{k^{\frac{p}{2}}}(\frac{p}{p+1})^{k},
 $$
 where $C_{1}=\sqrt{\pi}$, and $C_{p+1}=C_{p}\sqrt{\pi}(\frac{p+2}{p})^{\frac{p}{2}}$.
 \end{lemma}
 
 \begin{lemma}\label{COEF2}
 Define $C_{k,l}=\sum_{s=0}^{k-l}\binom{k}{s}(\frac{\lambda-l}{\lambda})^{l}A_{k-s,l}$. Then, as $k\rightarrow+\infty$,
  $$
 C_{k,l}\sim C_{l}\left(1+\frac{(\lambda-l)(l+1)}{\lambda l}\right)^{\frac{l}{2}}\frac{1}{k^\frac{l}{2}}\left(\frac{\lambda-l}{\lambda}+\frac{l}{l+1}\right)^{k}.
 $$
 \end{lemma}
 
 \begin{lemma}\label{COEF}
 For $\lambda>2$, define 
 $$
 K_{n}^{\lambda}(\theta)=\frac{\sum_{l=1}^{[\lambda]}\theta^{l}\sum_{k=l}^{\infty}\frac{(\lambda \log\frac{1}{\theta})^{k}}{k!}A_{k+n,l}(\theta)}{\sum_{l=1}^{[\lambda]}\theta^{l}\sum_{k=l}^{\infty}\frac{(\lambda \log\frac{1}{\theta})^{k}}{k!}A_{k,l}(\theta)}
 $$
 and 
 $$
 \tilde{K}_{n}^{\lambda}(\theta)=\frac{\sum_{l=1}^{[\lambda]}\theta^{l}\sum_{k=l}^{\infty}\frac{(\lambda \log\frac{1}{\theta})^{k}}{k!}A_{k+n,l}}{\sum_{l=1}^{[\lambda]}\theta^{l}\sum_{k=l}^{\infty}\frac{(\lambda \log\frac{1}{\theta})^{k}}{k!}A_{k,l}}.
 $$
When $u(u-1)<\lambda\leq u(u+1),\quad u\geq2$, we have
 $$
 \lim_{\theta\rightarrow0} K_{n}^{\lambda}(\theta)=\lim_{\theta\rightarrow0} \tilde{K}_{n}^{\lambda}(\theta)= (\frac{u-1}{u})^{n}.$$
 \end{lemma}
 
\textbf{[PROOF OF THEOREM \ref{Homozygosity_LLN}]:}

\begin{proof}
 Let us use $\phi_{H}$ to denote the moment generating function of the homozygosity H$_{2}$ under $\pi_{\lambda,\theta}$. Thus,
\begin{align*}
&\phi_{H}(t)=\frac{E\exp\{tH_{2}\}\cdot\exp\{\lambda\log\theta H_{2}\}}{E\exp\{\lambda\log\theta H_{2}\}}.
\end{align*}
For some technical reason,  we need to multiply the numerator and denumerator in the above equation by the common term $\lambda\log\frac{1}{\theta}$; then
\begin{align*}
\phi_{H}(t)=&e^{t}\frac{E\exp\{-t(1-H_{2})\}\cdot\exp\{\lambda\log\frac{1}{\theta}(1-H_{2})\}}
{E\exp\{\lambda\log\frac{1}{\theta}(1-H_{2})\}}\\
=&e^{t}\Bigg(1+\sum_{n=1}^{\infty}\frac{(-t)^{n}}{n!}
\frac{E(1-H_{2})^{n}\exp\{\lambda\log\frac{1}{\theta}(1-H_{2})\}}
{E\exp\{\lambda\log\frac{1}{\theta}(1-H_{2})\}}\Bigg)\\
=&e^{t}\Bigg(1+\sum_{n=1}^{\infty}\frac{(-t)^{n}}{n!}
\frac{\sum_{m=0}^{\infty}\frac{(\lambda\log\frac{1}{\theta})^{m}}{m!}E(1-H_{2})^{n+m}}
{\sum_{m=0}^{\infty}\frac{(\lambda\log\frac{1}{\theta})^{m}}{m!}E(1-H_{2})^{m}}\Bigg).
\end{align*} 
In the above expansion, all terms are positive, which greatly facilitates our calculations.
If we denote the limit of $\phi_{H}(t)$, as $\theta\rightarrow0$,  by $\psi_{H}(t)$, then we have
\begin{align*}
\psi_{H}(t)=e^{t}\Bigg[1+&\lim_{\theta\rightarrow0}\sum_{n=1}^{\infty}\frac{(-t)^{n}}{n!}\Bigg(\frac{\sum_{k=1}^{\infty}\frac{(\lambda\log\frac{1}{\theta})^{k}}{k!}\sum_{l=1}^{k}A_{k+n,l}(\theta)\theta^{l}}{1+\sum_{k=1}^{\infty}\frac{(\lambda\log\frac{1}{\theta})^{k}}{k!}\sum_{l=1}^{k}A_{k,l}(\theta)\theta^{l}}\\
+&\frac{m_{n}+\sum_{k=1}^{\infty}\frac{(\lambda\log\frac{1}{\theta})^{k}}{k!}\sum_{l=k+1}^{k+n}A_{k+n,l}(\theta)\theta^{l}}{1+\sum_{k=1}^{\infty}\frac{(\lambda\log\frac{1}{\theta})^{k}}{k!}\sum_{l=1}^{k}A_{k,l}(\theta)\theta^{l}}\Bigg)\Bigg].
\end{align*}
By the Lebesgue dominant convergence theorem, we can switch the order of summation and limit. Thus,
\begin{eqnarray}
\psi_{H}(t)=e^{t}\Bigg[1&+&\sum_{n=1}^{\infty}\frac{(-t)^{n}}{n!}\Bigg(\lim_{\theta\rightarrow0}\frac{\sum_{k=1}^{\infty}\frac{(\lambda\log\frac{1}{\theta})^{k}}{k!}\sum_{l=1}^{k}A_{k+n,l}(\theta)\theta^{l}}{1+\sum_{k=1}^{\infty}\frac{(\lambda\log\frac{1}{\theta})^{k}}{k!}\sum_{l=1}^{k}A_{k,l}(\theta)\theta^{l}}\label{afteralla}\\
&+&\lim_{\theta\rightarrow0}\frac{m_{n}+\sum_{k=1}^{\infty}\frac{(\lambda\log\frac{1}{\theta})^{k}}{k!}\sum_{l=k+1}^{k+n}A_{k+n,l}(\theta)\theta^{l}}{1+\sum_{k=1}^{\infty}\frac{(\lambda\log\frac{1}{\theta})^{k}}{k!}\sum_{l=1}^{k}A_{k,l}(\theta)\theta^{l}}\Bigg)\Bigg].\nonumber
\end{eqnarray}
Now we claim that 
\begin{align}
\lim_{\theta\rightarrow0}\frac{m_{n}+\sum_{k=1}^{\infty}\frac{(\lambda\log\frac{1}{\theta})^{k}}{k!}\sum_{l=1+k}^{k+n}A_{k+n,l}(\theta)\theta^{l}}{1+\sum_{k=1}^{\infty}\frac{(\lambda\log\frac{1}{\theta})^{k}}{k!}\sum_{l=1}^{k}A_{k,l}(\theta)\theta^{l}}=0.\label{claim_n_terms}
\end{align}
Indeed, we have
\begin{align*}
0\leq \frac{m_{n}+\sum_{k=1}^{\infty}\frac{(\lambda\log\frac{1}{\theta})^{k}}{k!}\sum_{l=1+k}^{k+n}A_{k+n,l}(\theta)\theta^{l}}{1+\sum_{k=1}^{\infty}\frac{(\lambda\log\frac{1}{\theta})^{k}}{k!}\sum_{l=1}^{k}A_{k,l}(\theta)\theta^{l}}\leq m_{n}+\sum_{k=1}^{\infty}\frac{(\lambda\log\frac{1}{\theta})^{k}}{k!}\sum_{l=1+k}^{k+n}A_{k+n,l}(\theta)\theta^{l}.
\end{align*}
By Lemma \ref{UBMH},  we have
\begin{align*}
&m_{n}+\sum_{k=1}^{\infty}\frac{(\lambda\log\frac{1}{\theta})^{k}}{k!}\sum_{l=1+k}^{k+n}A_{k+n,l}(\theta)\theta^{l}\leq m_{n}+\sum_{k=1}^{\infty}\frac{(\lambda\log\frac{1}{\theta})^{k}}{k!}\sum_{l=k+1}^{k+n}\frac{\theta^{l}}{2^{l-2}}\\
 \leq& m_{n}+4\sum_{k=1}^{\infty}\frac{(\lambda\frac{\theta}{2}\log\frac{1}{\theta})^{k}}{k!}(1-(\frac{\theta}{2})^{n})\frac{\theta}{2-\theta}\\
  =&m_{n}+4\frac{\theta}{2-\theta}(1-(\frac{\theta}{2})^{n})(e^{\lambda\frac{\theta}{2}\log\frac{1}{\theta}}-1)\rightarrow 0, \mbox{ as }\theta\rightarrow0.
\end{align*}
We have used the fact that $m_{n}\rightarrow0$, as $\theta\rightarrow0,$ which is due to $\mbox{PD}(\theta)(dx)\rightarrow\delta_{(1,0,\cdots)}(dx)$ as $\theta\rightarrow0$. Thus, claim (\ref{claim_n_terms}) is true.  Therefore, by switching the summation order in (\ref{afteralla}), we  have
$$
\psi_{H}(t)=e^{t}\Bigg(1+\sum_{n=1}^{\infty}\frac{(-t)^{n}}{n!}\lim_{\theta\rightarrow0}\frac{\sum_{l=1}^{\infty}\theta^{l}\sum_{k=l}^{\infty}\frac{(\lambda\log\frac{1}{\theta})^{k}}{k!}A_{k+n,l}(\theta)}{1+\sum_{l=1}^{\infty}\theta^{l}\sum_{k=l}^{\infty}\frac{(\lambda\log\frac{1}{\theta})^{k}}{k!}A_{k,l}(\theta)}\Bigg).
$$
Now what we need to show is, for $u(u-1)<\lambda\leq u(u+1), u\geq1$, 
\begin{align}
\lim_{\theta\rightarrow0}\frac{\sum_{l=1}^{\infty}\theta^{l}\sum_{k=l}^{\infty}\frac{(\lambda\log\frac{1}{\theta})^{k}}{k!}A_{k+n,l}(\theta)}{1+\sum_{l=1}^{\infty}\theta^{l}\sum_{k=l}^{\infty}\frac{(\lambda\log\frac{1}{\theta})^{k}}{k!}A_{k,l}(\theta)}=\left(\frac{u-1}{u}\right)^{n}.\label{CLAIM}
\end{align}
Once we have got the above equation, then
\begin{align*}
\psi_{H}(t)=&e^{t}\left(1+\sum_{n=1}^{\infty}\frac{(-t)^{n}}{n!}\Big(\frac{u-1}{u}\Big)^{n}\right)\\
=&e^{t}\left(1+\sum_{n=1}^{\infty}\frac{\Big(-\frac{t(u-1)}{u}\Big)^{n}}{n!}\right)=e^{t}e^{-\frac{t(u-1)}{u}}=e^{\frac{t}{u}}.
\end{align*}
Thus, $\psi_{H}(t)=\sum_{u=1}^{\infty}I_{(u(u-1),u(u+1)]}(\lambda)e^{\frac{t}{u}}$. Therefore, the limiting distribution of $H_{2}$ under $\pi_{\lambda,\theta}$ is $\sum_{u=1}^{\infty}I_{(u(u-1),u(u+1)]}(\lambda)\delta_{\frac{1}{u}}.$

Now we are going to verify the claim (\ref{CLAIM}). Firstly, when $0<\lambda\leq2$, we have
\begin{align*}
0\leq &\lim_{\theta\rightarrow0}\frac{\sum_{l=1}^{\infty}\theta^{l}\sum_{k=l}^{\infty}\frac{(\lambda\log\frac{1}{\theta})^{k}}{k!}A_{k+n,l}(\theta)}{1+\sum_{l=1}^{\infty}\theta^{l}\sum_{k=l}^{\infty}\frac{(\lambda\log\frac{1}{\theta})^{k}}{k!}A_{k,l}(\theta)}\leq \sum_{l=1}^{\infty}\theta^{l}\sum_{k=l}^{\infty}\frac{(\lambda\log\frac{1}{\theta})^{k}}{k!}A_{k+n,l}(\theta)\\
\leq&\theta \sum_{k=1}^{\infty}\frac{(\lambda\log\frac{1}{\theta})^{k}}{k!}A_{k+n,1}(\theta)+\theta^{2}\sum_{k=2}^{\infty}\frac{(\lambda\log\frac{1}{\theta})^{k}}{k!}A_{k+n,2}(\theta)\\
 &+\sum_{l=3}^{\infty}\theta^{l}\sum_{k=l}^{\infty}\frac{(\lambda\log\frac{1}{\theta})^{k}}{k!}A_{k+n,l}(\theta).
\end{align*}
We can actually show that the above three terms approach 0 as $\theta\rightarrow0$. Indeed, by Lemma \ref{UBMH}, we have
\begin{align*}
0\leq&\sum_{l=3}^{\infty}\theta^{l}\sum_{k=l}^{\infty}\frac{(\lambda\log\frac{1}{\theta})^{k}}{k!}A_{k+n,l}(\theta)
\leq\sum_{l=3}^{\infty}\theta^{l}\sum_{k=l}^{\infty}\frac{(\lambda\log\frac{1}{\theta})^{k}}{k!}\frac{1}{2^{l-2}}\\
\leq&4\sum_{l=3}^{\infty}(\frac{\theta}{2})^{l}e^{\lambda\log\frac{1}{\theta}}=\frac{4}{\theta^{\lambda}}(\frac{\theta}{2})^{3}\frac{1}{1-\theta/2}=\frac{\theta^{3-\lambda}}{2-\theta}\rightarrow0,\mbox{ as } \theta\rightarrow0.
\end{align*}
And by Lemma \ref{UBMH}, we have
\begin{align*}
0\leq &\theta \sum_{k=1}^{\infty}\frac{(\lambda\log\frac{1}{\theta})^{k}}{k!}A_{k+n,1}(\theta)\leq \theta \sum_{k=1}^{\infty}\frac{(\lambda\log\frac{1}{\theta})^{k}}{k!}A_{k+n,1}\\
\leq &\frac{\sum_{k=1}^{\infty}\frac{(\lambda\log\frac{1}{\theta})^{k}}{k!}A_{k+n,1}}{\sum_{k=0}^{\infty}\frac{(\log\frac{1}{\theta})^{k}}{k!}}\rightarrow0,\quad \mbox{ as } \theta\rightarrow0.
\end{align*}
The above limit is due to Lemma \ref{COEF1} and Lemma \ref{ML}. Similarly, 
\begin{align*}
0\leq &\theta^{2} \sum_{k=2}^{\infty}\frac{(\lambda\log\frac{1}{\theta})^{k}}{k!}A_{k+n,2}(\theta)\leq \theta^{2} \sum_{k=1}^{\infty}\frac{(\lambda\log\frac{1}{\theta})^{k}}{k!}A_{k+n,2}\\
\leq &\frac{\sum_{k=2}^{\infty}\frac{(\lambda\log\frac{1}{\theta})^{k}}{k!}A_{k+n,2}}{\sum_{k=0}^{\infty}\frac{(2\log\frac{1}{\theta})^{k}}{k!}}\rightarrow0,\quad \mbox{ as } \theta\rightarrow0.
\end{align*}
Thus, we have for $0<\lambda\leq 2$,
$$
\lim_{\theta\rightarrow0}\frac{\sum_{l=1}^{\infty}\theta^{l}\sum_{k=l}^{\infty}\frac{(\lambda\log\frac{1}{\theta})^{k}}{k!}A_{k+n,l}(\theta)}{1+\sum_{l=1}^{\infty}\theta^{l}\sum_{k=l}^{\infty}\frac{(\lambda\log\frac{1}{\theta})^{k}}{k!}A_{k,l}(\theta)}=0=(\frac{1-1}{1})^{n}.
$$
Secondly, for $u(u-1)<\lambda\leq u(u+1), u\geq2$, then $\lambda>2$.  we can show that 
\begin{align}
\lim_{\theta\rightarrow0}\frac{\sum_{l=1}^{\infty}\theta^{l}\sum_{k=l}^{\infty}\frac{(\lambda\log\frac{1}{\theta})^{k}}{k!}A_{k+n,l}(\theta)}{1+\sum_{l=1}^{\infty}\theta^{l}\sum_{k=l}^{\infty}\frac{(\lambda\log\frac{1}{\theta})^{k}}{k!}A_{k,l}(\theta)}=\lim_{\theta\rightarrow0}K_{n}^{\lambda}(\theta);\label{CLAIMK}
\end{align}
then by Lemma \ref{COEF}, we have proved claim (\ref{CLAIM}).  Now we only need to verify (\ref{CLAIMK}).  To this end, we rewrite
$$
\frac{\sum_{l=1}^{\infty}\theta^{l}\sum_{k=l}^{\infty}\frac{(\lambda\log\frac{1}{\theta})^{k}}{k!}A_{k+n,l}(\theta)}{1+\sum_{l=1}^{\infty}\theta^{l}\sum_{k=l}^{\infty}\frac{(\lambda\log\frac{1}{\theta})^{k}}{k!}A_{k,l}(\theta)}
$$
as
$$
\frac{\sum_{l=1}^{[\lambda]}\theta^{l}\sum_{k=l}^{\infty}\frac{(\lambda\log\frac{1}{\theta})^{k}}{k!}A_{k+n,l}(\theta)+\sum_{l=1+[\lambda]}^{\infty}\theta^{l}\sum_{k=l}^{\infty}\frac{(\lambda\log\frac{1}{\theta})^{k}}{k!}A_{k+n,l}(\theta)}{1+\sum_{l=1}^{[\lambda]}\theta^{l}\sum_{k=l}^{\infty}\frac{(\lambda\log\frac{1}{\theta})^{k}}{k!}A_{k,l}(\theta)+\sum_{l=1+[\lambda]}^{\infty}\theta^{l}\sum_{k=l}^{\infty}\frac{(\lambda\log\frac{1}{\theta})^{k}}{k!}A_{k,l}(\theta)}.
$$
 By Lemma \ref{FT}, we know, as $\theta\rightarrow0$
$$
\frac{\sum_{l=1+[\lambda]}^{\infty}\theta^{l}\sum_{k=l}^{\infty}\frac{(\lambda\log\frac{1}{\theta})^{k}}{k!}A_{k+n,l}(\theta)}{1+\sum_{l=1}^{[\lambda]}\theta^{l}\sum_{k=l}^{\infty}\frac{(\lambda\log\frac{1}{\theta})^{k}}{k!}A_{k,l}(\theta)}\rightarrow0,
$$
and
$$
\frac{\sum_{l=1+[\lambda]}^{\infty}\theta^{l}\sum_{k=l}^{\infty}\frac{(\lambda\log\frac{1}{\theta})^{k}}{k!}A_{k,l}(\theta)}{1+\sum_{l=1}^{[\lambda]}\theta^{l}\sum_{k=l}^{\infty}\frac{(\lambda\log\frac{1}{\theta})^{k}}{k!}A_{k,l}(\theta)}\rightarrow0.
$$
Therefore, claim (\ref{CLAIMK}) is proved. Theorem \ref{Homozygosity_LLN} is thus proved!
\end{proof}

Now we are ready to show the weak law of large numbers for $\pi_{\lambda,\theta}$. Define $d(\cdot,\cdot)$ as
 $$
 d(x,y)=\sum_{i=1}^{\infty}\frac{|x_{i}-y_{i}|}{2^{i}}.
 $$
 Under $d(\cdot,\cdot)$, $(\bar{\triangledown}_{\infty},d)$ is a compact space, and the LDP for $\pi_{\lambda,\theta}$ was originally proved under the metric $d(\cdot,\cdot)$ in \cite{MR2519357}. Denote $C(\bar{\triangledown}_{\infty})$ to be the set of all continuous functions in $(\bar{\triangledown}_{\infty},d)$.
 
 \begin{theorem}\label{good_stuff}
 For a given $\lambda>0$, $\pi_{\lambda,\theta}$ converges weakly to the following
$$
\sum_{k=1}^{\infty}{I}_{(k(k-1),k(k+1)]}(\lambda)\delta_{(\frac{1}{k},\cdots,\frac{1}{k},0,\cdots)}.
$$
\end{theorem}

\begin{proof} 
For integer $k\geq1$, and $(k-1)k<\lambda<k(k+1)$, the limit of $\pi_{\lambda,\theta}$ has already been verified in \cite{MR2519357}. Therefore, we only need to consider the critical case $\lambda=(k+1)k, k\geq1$. For a fixed $k\geq 1$, and $\lambda=(k+1)k$, we need to prove that $\pi_{\lambda,\theta}$ converges weakly to $\delta_{(\frac{1}{k},\cdots,\frac{1}{k},0,\cdots)}$.

Define $B_{\delta}(\frac{1}{k},\cdots,\frac{1}{k},0,\cdots)$ to be a ball with
center $(\frac{1}{k},\cdots,\frac{1}{k},0,\cdots)$ and radius $\delta$.
For a given $f\in C(\bar{\triangledown}_{\infty})$, one can conclude that, $\forall\epsilon>0,\exists \delta<\frac{1}{k(k+1)+1},$ such that \\
$\forall x\in B_{\delta}(\frac{1}{k},\cdots,\frac{1}{k},0,\cdots)$, 
$$
\Big|f(x)-f\left(\frac{1}{k},\cdots,\frac{1}{k},0,\cdots\right)\Big|<\epsilon.
$$
Thus,
\begin{align*}
&\Big|\int_{\triangledown_{\infty}}f(x)\pi_{\lambda,\theta}(dx)-f\left(\frac{1}{k},\cdots,\frac{1}{k},0,\cdots\right)\Big|\\
=&\Big|\int_{\triangledown_{\infty}}\left(f(x)-f\left(\frac{1}{k},\cdots,\frac{1}{k},0,\cdots\right)\right)\pi_{\lambda,\theta}(dx)\Big|\\
\leq &\int_{\triangledown_{\infty}}\Big|f(x)-f\left(\frac{1}{k},\cdots,\frac{1}{k},0,\cdots\right)\Big|\pi_{\lambda,\theta}(dx)\\
=&\int_{S_{\lambda}\geq\delta}\bigg|f(x)-f\left(\frac{1}{k},\cdots,\frac{1}{k},0,\cdots\right)\bigg|\pi_{\lambda,\theta}(dx)\\
&+\int_{S_{\lambda}<\delta}\bigg|f(x)-f\left(\frac{1}{k},\cdots,\frac{1}{k},0,\cdots\right)\bigg|\pi_{\lambda,\theta}(dx)\\
\leq& 2\|f\|_{\infty}\pi_{\lambda,\theta}(S_{\lambda}\geq\delta)+\int_{(S_{\lambda}<\delta)\cap(|\varphi_{2}-\frac{1}{k}|\geq\delta)}\bigg|f(x)-f\left(\frac{1}{k},\cdots,\frac{1}{k},0,\cdots\right)\bigg|\pi_{\lambda,\theta}(dx)\\
&+\int_{(S_{\lambda}<\delta)\cap(|\varphi_{2}-\frac{1}{k}|<\delta)}\bigg|f(x)-f\left(\frac{1}{k},\cdots,\frac{1}{k},0,\cdots\right)\bigg|\pi_{\lambda,\theta}(dx)\\
\leq& 2\|f\|_{\infty}\bigg[\pi_{\lambda,\theta}(S_{\lambda}\geq\delta)+\pi_{\lambda,\theta}\left(\bigg|\varphi_{2}-\frac{1}{k}\bigg|\geq\delta\right)\bigg]\\
&+\int_{(S_{\lambda}<\delta)\cap(|\varphi_{2}-\frac{1}{k}|<\delta)}\bigg|f(x)-f\left(\frac{1}{k},\cdots,\frac{1}{k},0,\cdots\right)\bigg|\pi_{\lambda,\theta}(dx)
\end{align*}
By the LDP for $\pi_{\lambda,\theta}$ and the weak convergence of $H_{2}$ under $\pi_{\lambda,\theta}$, we have
$$
\lim_{\theta\rightarrow0}\pi_{\lambda,\theta}(S_{\lambda}\geq\delta)=\lim_{\theta\rightarrow0}\pi_{\lambda,\theta}\left(\bigg|\varphi_{2}-\frac{1}{k}\bigg|\geq\delta\right)=0.
$$
Moreover, we claim that 
\begin{align}
(S_{\lambda}<\delta)\cap\left(\bigg|\varphi_{2}-\frac{1}{k}\bigg|<\delta\right)\subset
B_{\delta}\left(\frac{1}{k},\cdots,\frac{1}{k},0,\cdots\right).\label{inclusion}
\end{align}
Then we have
$$
\limsup_{\theta\rightarrow0}\bigg|\int_{\triangledown_{\infty}}f(x)\pi_{\lambda,\theta}(dx)-f\left(\frac{1}{k},\cdots,\frac{1}{k},0,\cdots\right)\bigg|\leq \epsilon.
$$
Letting $\epsilon\rightarrow0$, we have
$$
\lim_{\theta\rightarrow0}\int_{\triangledown_{\infty}}f(x)\pi_{\lambda,\theta}(dx)=f\left(\frac{1}{k},\cdots,\frac{1}{k},0,\cdots\right).
$$
Therefore, $\pi_{\lambda,\theta}$ converges weakly to $\delta_{(\frac{1}{k},\cdots,\frac{1}{k},0,\cdots)}$.
Now we need to show the claim (\ref{inclusion}). For $\lambda=(k+1)k$, since
$$
S_{\lambda}(x)=J(x)+(k+1)k\varphi_{2}(x)-\inf_{n\geq1}\left\{\frac{(k+1)k}{n}+n-1\right\}, 
$$
clearly
$$
S_{\lambda}|_{L_{n}}(x)=n-1+k(k+1)\varphi_{2}|_{L_{n}}(x)-2k.
$$
Because $\varphi_{2}|_{L_{n}}(x)$ has a unique minimum point $(\frac{1}{n},\cdots,\frac{1}{n},0,\cdots)$; then
$$
S_{\lambda}\Big|_{L_{n}}(x)\geq n-1+\frac{k(k+1)}{n}-2k.
$$
By the monotonicity of the righthand function in $n$, we know it attains its minimum at $k$ and $k+1$. Since $\delta<\frac{1}{k(k+1)+1}<\frac{2}{k+2},$ one can see, $\forall n\neq k,k+1$,
$$
S_{\lambda}|_{L_{n}}(x)\geq\min\left\{k-1-1+\frac{k(k+1)}{k-1}-2k,k+2-1+\frac{k(k+1)}{k+2}-2k\right\}=\frac{2}{k+2}>\delta.
$$
Then $(S_{\lambda}<\delta)=(S_{\lambda}<\delta)\cap (L_{k}\cup L_{k+1})$. Thus, 
\begin{align*}
&(S_{\lambda}<\delta)\cap\left(\bigg|\varphi_{2}-\frac{1}{k}\bigg|<\delta\right)\\
=&\Bigg[(S_{\lambda}<\delta)\cap\left(\bigg|\varphi_{2}-\frac{1}{k}\bigg|<\delta\right)\cap L_{k}\Big]\cup\Bigg[(S_{\lambda}<\delta)\cap\left(\bigg|\varphi_{2}-\frac{1}{k}\bigg|<\delta\right)\cap L_{k+1}\Bigg].
\end{align*}
But $S_{\lambda}|_{L_{k+1}}=k(k+1)(\varphi_{2}|_{L_{k+1}}-\frac{1}{k+1})$, then $(S_{\lambda}<\delta)\cap L_{k+1}=(\varphi_{2}<\frac{\delta}{k(k+1)}+\frac{1}{k+1})\cap L_{k+1}$. Since $\delta<\frac{1}{k(k+1)+1}$, then $\frac{\delta}{k(k+1)}+\frac{1}{k+1}<\frac{1}{k}-\delta$, thus  $\forall x\in L_{k+1}\cap (|\varphi_{2}-\frac{1}{k}|<\delta)\cap (S_{\lambda}<\delta)$, we have
$$
\varphi_{2}(x)<\frac{\delta}{k(k+1)}+\frac{1}{k+1}<\frac{1}{k}-\delta<\varphi_{2}(x).
$$
Therefore, $(S_{\lambda}<\delta)\cap(|\varphi_{2}-\frac{1}{k}|<\delta)\cap L_{k+1}=\emptyset$, hence,
$$
(S_{\lambda}<\delta)\cap\left(\bigg|\varphi_{2}-\frac{1}{k}\bigg|<\delta\right)=(S_{\lambda}<\delta)\cap\left(\bigg|\varphi_{2}-\frac{1}{k}\bigg|<\delta\right)\cap L_{k}.
$$
Since $\forall x\in L_{k}\cap\left(|\varphi_{2}(x)-\frac{1}{k}|<\delta\right)$, we have
$\frac{1}{k}-\delta<\sum_{i=1}^{k}x_{i}^{2}<\frac{1}{k}+\delta$,  and
$$
\frac{1}{k^{2}}-\frac{\delta}{k}\leq \frac{\sum_{i=1}^{k}x_{i}^{2}}{k}<\frac{1}{k^{2}}+\frac{\delta}{k}.
$$
Therefore, $\sqrt{\frac{1}{k^{2}}-\frac{\delta}{k}}<\min_{1\leq i\leq k}x_{i}\leq \max_{1\leq i\leq k}x_{i}\leq \sqrt{\frac{1}{k^{2}}+\frac{\delta}{k}}$.
Then
\begin{align*}
&d\left(x,\left(\frac{1}{k},\cdots,\frac{1}{k},0,\cdots\right)\right)=\sum_{i=1}^{k}\frac{|x_{i}-1/k|}{2^{i}}\\
\leq &\sum_{i=1}^{k}\frac{1}{2^{i}} \max_{1\leq i\leq k}\big|x_{i}-\frac{1}{k}\big|=\left(1-\frac{1}{2^{k}}\right)\max_{1\leq i\leq k}\big|x_{i}-\frac{1}{k}\big|\\
\leq&\max\left\{\bigg|\sqrt{\frac{1}{k^{2}}-\frac{\delta}{k}}-\frac{1}{k}\bigg|,\bigg|\sqrt{\frac{1}{k^{2}}+\frac{\delta}{k}}-\frac{1}{k}\bigg|\right\}\\
=&\max\left\{\frac{\frac{\delta}{k}}{\sqrt{\frac{1}{k^{2}}-\frac{\delta}{k}}+\frac{1}{k}},\frac{\frac{\delta}{k}}{\sqrt{\frac{1}{k^{2}}+\frac{\delta}{k}}+\frac{1}{k}}\right\}\\
<&\delta.
\end{align*}
Therefore $(S_{\lambda}<\delta)\cap(|\varphi_{2}-\frac{1}{k}|<\delta)\subset B_{\delta}(\frac{1}{k},\cdots,\frac{1}{k},0,\cdots)$, the claim (\ref{inclusion}) is thus proved.
\end{proof}

\section{Proof of Lemmas}

 The LDP estimations of binomial distribution and negative binomial distribution will be needed in this section. Hence, example \ref{LDP_EXAM} is presented.
 
\begin{exm}\label{LDP_EXAM}
Let $X_{\alpha}^{k}=\sum_{l=1}^{k} Y_{l}^{\alpha},$ and $U_{\alpha}^{k}=\sum_{l=1}^{k}V_{l}^{\alpha}$, where $\{Y_{l}^{\alpha},1\leq l\leq k\}$ and $\{V_{l}^{\alpha},1\leq l\leq k\}$ are i.i.d. geometric random variables and Bernoulli random variables respectively; i.e.
$$
P(Y_{l}^{\alpha}=u)=(1-\alpha)^{u}\alpha,u\geq0; \quad P(V_{l}^{\alpha}=v)=\alpha^{v}(1-\alpha)^{1-v},v=0,1.
$$ 
Then the distributions of $X_{\alpha}^{k}$ and $U_{\alpha}^{k}$, denoted by $\mu_{k}$ and $\nu_{k}$, satisfy LDPs with speed $k$ and rate function $I_{1}(x)$ and $I_{2}(x)$ respectively, where
$$
I_{1}(x)=x\log x-(x+1)\log(1+x)-[x\log (1-\alpha)+\log\alpha]
$$
and
$$
I_{2}(x)=x\log\left(\frac{x}{\alpha}\right)+(1-x)\log\left(\frac{1-x}{1-\alpha}\right).
$$
\end{exm}
\begin{proof} 
This can be shown by the Cram\'{e}r theorem. Please refer to \\
\cite{MR2571413}.
\end{proof}

Next, we will embark on a long journey to prove the previous lemmas.

\textbf{[PROOF OF LEMMA \ref{MH}]:}
\begin{proof}
 Define $V_{1}=U_{1},V_{i}=(1-U_{1})\cdots(1-U_{i-1})U_{i},i\geq2,$ where $\{U_{i},i\geq1\}$ are i.i.d. Beta(1,$\theta$). Then
$(V_{1},V_{2},\cdots)$ follows the GEM distribution. Denote $(V_{(1)},V_{(2)},\cdots)$ to be the descending statistics of $(V_{1},V_{2},\cdots)$; then $(V_{(1)},V_{(2)},\cdots)$ follows the Poisson-Dirichlet distribution PD($\theta$). Since $H_{2}=\sum_{i=1}^{\infty}V_{(i)}^{2}=\sum_{i=1}^{\infty}V_{i}^{2},$ one can observe that $1-H_{2}=(1-U_{1}^{2})-(1-U_{1})^{2}+(1-U_{1})^{2}(1-\tilde{H}_{2})=2U_{1}(1-U_{1})+(1-U_{1})^{2}(1-\tilde{H}_{2})$, where $\tilde{H}_{2}=\sum_{i=1}^{\infty}\tilde{V}_{i}^{2}$, and $\tilde{V}_{1}=U_{2}, \tilde{V}_{i}=(1-U_{2})\cdots(1-U_{i})U_{i+1}, i\geq2$. We can see that $(\tilde{V}_{1},\tilde{V}_{2},\cdots)$ follows the GEM distribution as well and is independent of $U_{1}$. Thus,
$E(1-H_{2})^{k}=E(1-\tilde{H}_{2})^{k},$ and 
\begin{align*}
m_{k}=&E(1-H_{2})^{k}=E\left(2U_{1}(1-U_{1})+(1-U_{1})^{2}(1-\tilde{H}_{2})\right)^{k}\\
=&E\sum_{l=0}^{k}\binom{k}{l}\left(2U_{1}(1-U_{1})\right)^{k-l}(1-\tilde{H}_{2})^{l}\\
=&\sum_{l=0}^{k}\binom{k}{l}E\left(2U_{1}(1-U_{1})\right)^{k-l}E(1-\tilde{H}_{2})^{l}\\
=&\sum_{l=0}^{k}\binom{k}{l}\frac{2^{k-l}\Gamma(k-l+1)\Gamma(k+l+\theta)\theta}{\Gamma(2k+1+\theta)}m_{l}.
\end{align*}
 If we isolate $m_{k}$, we have 
 \begin{align}
 m_{k}=\theta\sum_{l=1}^{k-1}\frac{2k+\theta}{2k}\frac{2^{k}k!}{2^{l}l!}\frac{\Gamma(k+l+\theta)}{\Gamma(2k+1+\theta)}m_{l}+\frac{2k+\theta}{2k}\frac{2^{k}k!\Gamma(k+\theta)}{\Gamma(2k+1+\theta)}\theta,\label{RCM}
 \end{align}
 where $k\geq2$, and $m_{1}=\frac{\theta}{1+\theta}.$ We claim that $m_{k}$ has the following expansion
 $$
 m_{k}=\sum_{l=1}^{k}A_{k,l}(\theta)\theta^{l},
 $$
 where 
 \begin{align*}
 A_{k,l}(\theta)&=\sum_{u=l-1}^{k-1}\frac{2k+\theta}{2k}\frac{2^{k}k!}{2^{u}u!}\frac{\Gamma(k+u+\theta)}{\Gamma(2k+1+\theta)}A_{u,l-1}(\theta),l\geq2;\\
  A_{k,1}(\theta)&=\frac{2k+\theta}{2k}\frac{2^{k}k!\Gamma(k+\theta)}{\Gamma(2k+\theta+1)}.
 \end{align*}
 Indeed, for $k=1$, this is obvious. Assume that $m_{k-1}$ has the above expression, then for $m_{k}$, by (\ref{RCM}), we have
 \begin{align*}
 &\theta\sum_{l=1}^{k-1}\frac{2k+\theta}{2k}\frac{2^{k}k!}{2^{l}l!}\frac{\Gamma(k+l+\theta)}{\Gamma(2k+1+\theta)}\sum_{u=1}^{l}A_{l,u}(\theta)\theta^{u}+\frac{2k+\theta}{2k}\frac{2^{k}k!\Gamma(k+\theta)}{\Gamma(2k+1+\theta)}\theta \\
 =&\sum_{u=1}^{k-1}\left[\sum_{l=u}^{k-1}\frac{2k+\theta}{2k}\frac{2^{k}k!}{2^{l}l!}\frac{\Gamma(k+l+\theta)}{\Gamma(2k+1+\theta)}A_{l,u}(\theta)\right]\theta^{u+1}+\frac{2k+\theta}{2k}\frac{2^{k}k!\Gamma(k+\theta)}{\Gamma(2k+1+\theta)}\theta.
 \end{align*}
 Let us denote $p=u+1$; then we have $m_{k}=\sum_{p=1}^{k}A_{k,p}(\theta)\theta^{p},$ where 
 \begin{align*}
 A_{k,1}(\theta)=&\frac{2k+\theta}{2k}\frac{2^{k}k!\Gamma(k+\theta)}{\Gamma(2k+1+\theta)},\\
 A_{k,p}(\theta)=&\sum_{l=p-1}^{k-1}\frac{2k+\theta}{2k}\frac{2^{k}k!}{2^{l}l!}\frac{\Gamma(k+l+\theta)}{\Gamma(2k+1+\theta)}A_{l,p-1}(\theta), \quad p\geq2.
 \end{align*}
 \end{proof}
 
 \textbf{[PROOF OF LEMMA \ref{UBMH}]:}
 
 \begin{proof}
 We use mathematical induction with respect to $p$ to show these conclusions. \\
 \textbf{Step1: } We are going to show $\frac{1}{2^{p+1}}A_{k,p+1}\leq A_{k,p+1}(\theta)\leq A_{k,p+1}.$\\
  When 
 $p=1$, we have, $\forall\theta\in[0,1]$,
  \begin{align*}
 A_{k,1}(\theta)=&\frac{2k+\theta}{2k}\frac{2^{k}k!}{(2k+\theta)\cdots(k+\theta)}=\frac{2^{k-1}(k-1)!}{(2k-1+\theta)\cdots(k+\theta)}\\
 \leq &\frac{2^{k-1}(k-1)!}{(2k-1)\cdots k}=\frac{2^{k}(k-1)!k!}{(2k)!}=A_{k,1},
 \end{align*}
 and 
 $$
 A_{k,1}(\theta)=\frac{2^{k-1}(k-1)!}{(2k-1+\theta)\cdots(k+\theta)}\geq\frac{2^{k-1}(k-1)!}{2k\cdots(k+1)}=\frac{1}{2}A_{k,1}.
 $$
 Moreover, $A_{k,1}=\frac{2^{k}k!(k-1)!}{(2k)!}\leq1<\frac{1}{2^{1-2}}=2$; therefore,
 $$
 \frac{1}{2}A_{k,1}\leq A_{k,1}(\theta)\leq A_{k,1}\leq\frac{1}{2^{1-2}}.
 $$
 Now we assume that $A_{k,p}(\theta)$ satisfies the inequality
 \begin{align}
 \frac{1}{2^{p}}A_{k,p}\leq A_{k,p}(\theta)\leq A_{k,p}\leq \frac{1}{2^{p-2}}, k\geq p\geq1.\label{ASSUMP}
 \end{align}
 Since $A_{k,p+1}(\theta)=\sum_{l=p}^{k-1}\frac{1}{2k}\frac{2^{k}k!}{2^{l}l!}\frac{A_{l,p}(\theta)}{(2k-1+\theta)\cdots(k+l+\theta)}$, by assumption (\ref{ASSUMP}), we have
\begin{align*}
 \frac{1}{2^{p}}\sum^{k-1}_{l=p}\frac{1}{2k}\frac{2^{k}k!}{2^{l}l!}\frac{A_{l,p}}{(2k-1+1)\cdots(k+l+1)} \leq &A_{k,p+1}(\theta)\\
 \leq&\sum_{l=p}^{k-1}\frac{1}{2k}\frac{2^{k}k!}{2^{l}l!}\frac{A_{l,p}}{(2k-1)\cdots(k+l)}.
\end{align*}
 Thus,
 $$
\frac{1}{2^{p}}\sum_{l=p}^{k-1}\frac{k+l}{2k}\frac{2^{k}k!}{2^{l}l!}\frac{\Gamma(k+l)}{\Gamma(2k+1)}A_{l,p}\leq A_{k,p+1}(\theta)\leq \sum_{l=p}^{k-1}\frac{2^{k}k!}{2^{l}l!}\frac{\Gamma(k+l)}{\Gamma(2k+1)}A_{l,p}=A_{k,p+1}.
$$
 But $\frac{k+l}{2k}>\frac{k}{2k}=\frac{1}{2}$, then 
 $$
 \frac{1}{2^{p}}\sum_{l=p}^{k-1}\frac{k+l}{2k}\frac{2^{k}k!}{2^{l}l!}\frac{\Gamma(k+l)}{\Gamma(2k+1)}A_{l,p}
 \geq \frac{1}{2^{p+1}} \sum_{l=p}^{k-1}\frac{2^{k}k!}{2^{l}l!}\frac{\Gamma(k+l)}{\Gamma(2k+1)}A_{l,p}=\frac{1}{2^{p+1}}A_{k,p+1}.
 $$
 Hence, 
 $$
 \frac{1}{2^{p+1}}A_{k,p+1}\leq A_{k,p+1}(\theta)\leq A_{k,p+1}.
 $$
 \textbf{Step 2: } We are going to show  
 \begin{align}
 A_{k,p}\leq \frac{1}{2^{p-2}},\forall k\geq p\geq1. \label{uuuper}
 \end{align}
 Since $A_{k,1}\leq\frac{1}{2^{1-2}}$, we can assume that $A_{k,p}\leq\frac{1}{2^{p-2}},k\geq p$. So 
 $$
 A_{k,p+1}=\sum_{l=p}^{k-1}\frac{2^{k}k!}{2^{l}l!}\frac{\Gamma(k+l)}{\Gamma(2k+1)}A_{l,p}\leq \frac{1}{2^{p-2}}\sum_{l=p}^{k-1}\frac{2^{k}k!\Gamma(k+l)}{2^{l}l!\Gamma(2k+1)}, 
 $$
 where we will show that $\sum_{l=p}^{k-1}\frac{2^{k}k!\Gamma(k+l)}{2^{l}l!\Gamma(2k+1)}<\frac{1}{2}.$ Therefore, $A_{k,p+1}\leq\frac{1}{2^{p-1}},$ and (\ref{uuuper}) is thus proved. Next, to show 
$\sum_{l=p}^{k-1}\frac{2^{k}k!\Gamma(k+l)}{2^{l}l!\Gamma(2k+1)}<\frac{1}{2},$  let us define $B_{k,l}=\frac{2^{k}k!\Gamma(k+l)}{2^{l}l!\Gamma(2k+1)},\quad k-1\geq l\geq p\geq2,$ where $B_{k,l}$ is increasing in $l$ only if $l<k-2$. Because when $l<k-2$
$$
 \frac{B_{k,l+1}}{B_{k,l}}=\frac{2^{l}l!\Gamma(k+l+1)}{2^{l+1}(l+1)!\Gamma(k+l)}=\frac{k+l}{2(l+1)}>1.
$$
Therefore, $B_{k,k-2}$ or $B_{k,k-1}$ should be the maximum term. Since
  $$
  B_{k,k-2}=\frac{2^{k}k!}{2^{k-2}(k-2)!}\frac{\Gamma(2k-2)}{\Gamma(2k+1)}=\frac{1}{2k-1},
  $$
  $$
  B_{k,k-1}=\frac{2^{k}k!}{2^{k-1}(k-1)!}\frac{\Gamma(2k-1)}{\Gamma(2k+1)}=\frac{1}{2k-1},  
$$
 we obtain that 
  $$
  \sum_{l=p}^{k-1}\frac{2^{k}k!\Gamma(k+l)}{2^{l}l!\Gamma(2k+1)}\leq \frac{k-p}{2k-1}=\frac{k-p}{2(k-p)+2p-1}<\frac{1}{2}.
  $$
  \textbf{Step 3:} We are going to show
  $$
  \Big|A_{k,p}(\theta)-A_{k,p}\Big|\leq\theta pA_{k,p}.
  $$
  When $p=1$, $ \Big|A_{k,1}(\theta)-A_{k,1}\Big|\leq\theta A_{k,1}$, for
 \begin{align*}
 \Big|A_{k,1}(\theta)-A_{k,1}\Big|=&\Big|\frac{2^{k}k!\Gamma(k+\theta)}{2k\Gamma(2k+\theta)}-\frac{2^{k}k!\Gamma(k)}{\Gamma(2k+1)}\Big|\\
 =&\Big|\frac{2^{k}k!\Gamma(k)}{\Gamma(2k+1)}\Big(\frac{\Gamma(2k)\Gamma(k+\theta)}{\Gamma(2k+\theta)\Gamma(k)}-1\Big)\Big|\\
  \leq &A_{k,1}\Big|\frac{(2k-1)\cdots k}{(2k+\theta-1)\cdots (k+\theta)}-1\Big|\\
 =&A_{k,1}\Big|(1-\frac{\theta}{2k+\theta-1})\cdots(1-\frac{\theta}{k+\theta})-1\Big|\\
  =&\theta A_{k,1}\Big|\sum_{u=1}^{k}(-1)^{u}\sum_{k\leq l_{1}<\cdots<l_{u}\leq 2k-1}\frac{\theta^{u-1}}{(l_{1}+\theta)\cdots(l_{u}+\theta)}\Big|\\
 \leq&\theta A_{k,1}\sum_{u=1}^{k}\sum_{k\leq l_{1}<\cdots<l_{u}\leq 2k-1}\frac{\theta^{u-1}}{(l_{1}+\theta)\cdots(l_{u}+\theta)}\\
  \leq& \theta A_{k,1}\sum_{u=1}^{k}\sum_{k\leq l_{1}<\cdots<l_{u}\leq 2k-1}\frac{1}{l_{1}\cdots l_{u}}\\
 =&\theta A_{k,1}\Big|(1+\frac{1}{2k-1})\cdots(1+\frac{1}{k})-1\Big|\\
 =&\theta A_{k,1}\Big|\frac{2k}{2k-1}\frac{2k-1}{2k-2}\cdots\frac{k+1}{k}-1\Big|=\theta A_{k,1}|2-1|=\theta A_{k,1}.
 \end{align*}
Therefore we assume that 
  \begin{align}
  \Big|A_{k,p}(\theta)-A_{k,p}\Big|\leq\theta pA_{k,p},\label{CLAIM2}
  \end{align}
  then, for $\Big|A_{k,p+1}(\theta)-A_{k,p+1}\Big|,$ we have
  \begin{align*}
  \Big|A_{k,p+1}(\theta)-A_{k,p+1}\Big|
  =&\Big|\sum_{l=p}^{k-1}\frac{2k+\theta}{2k}\frac{2^{k}k!\Gamma(k+l+\theta)}{2^{l}l!\Gamma(2k+1+\theta)}A_{l,p}(\theta)-\sum_{l=p}^{k-1}\frac{2^{k}k!\Gamma(k+l)}{2^{l}l!\Gamma(2k+1)}A_{l,p}\Big|\\
  \leq&\Big|\sum_{l=p}^{k-1}\frac{2k+\theta}{2k}\frac{2^{k}k!\Gamma(k+l+\theta)}{2^{l}l!\Gamma(2k+1+\theta)}(A_{l,p}(\theta)-A_{l,p})\Big|\\
 &+\Big|\sum_{l=p}^{k-1}\frac{2k+\theta}{2k}\frac{2^{k}k!\Gamma(k+l+\theta)}{2^{l}l!\Gamma(2k+1+\theta)}A_{l,p}-\sum_{l=p}^{k-1}\frac{2^{k}k!\Gamma(k+l)}{2^{l}l!\Gamma(2k+1)}A_{l,p}\Big|\\
 \leq&\sum_{l=p}^{k-1}\frac{2k+\theta}{2k}\frac{2^{k}k!\Gamma(k+l+\theta)}{2^{l}l!\Gamma(2k+1+\theta)}\Big|A_{l,p}(\theta)-A_{l,p}\Big|  \\
 &+\sum_{l=p}^{k-1}\frac{2^{k}k!\Gamma(k+l)}{2^{l}l!\Gamma(2k+1)} A_{l,p}\Big|\frac{\Gamma(k+l+\theta)\Gamma(2k)}{\Gamma(k+l)\Gamma(2k+\theta)}-1\Big|. 
 \end{align*}
  By the assumption (\ref{CLAIM2}), we have
  \begin{align*}
   \Big|A_{k,p+1}(\theta)-A_{k,p+1}\Big|
  \leq& \theta p \sum_{l=p}^{k-1}\frac{2^{k}k!\Gamma(k+l+\theta)}{2^{l}l!2k\Gamma(2k+\theta)}A_{l,p}\\
  &+\sum_{l=p}^{k-1}\frac{2^{k}k!\Gamma(k+l)}{2^{l}l!\Gamma(2k+1)}A_{l,p}\Big|\frac{\Gamma(k+l+\theta)\Gamma(2k)}{\Gamma(2k+\theta)\Gamma(k+l)}-1\Big|,
  \end{align*}
 where
  \begin{align*}
  \Big|\frac{\Gamma(k+l+\theta)\Gamma(2k)}{\Gamma(2k+\theta)\Gamma(k+l)}-1\Big|=&\Big|\frac{(2k-1)\cdots(k+l)}{(2k+\theta-1)\cdots(k+l+\theta)}-1\Big|\\
  =&\Big|(1-\frac{\theta}{2k-1+\theta})\cdots(1-\frac{\theta}{k+l+\theta})-1\Big|\\
  =&\Big|\sum_{u=1}^{k-l}(-1)^{u}\sum_{k+l\leq l_{1}<\cdots<l_{u}\leq 2k-1}\frac{\theta^{u}}{(l_{1}+\theta)\cdots(l_{u}+\theta)}\Big|\\
  \leq&\sum_{u=1}^{k-l}\sum_{k+l\leq l_{1}<\cdots<l_{u}\leq 2k-1}\frac{\theta^{u}}{(l_{1}+\theta)\cdots(l_{u}+\theta)}\\
    \leq&\theta \sum_{u=1}^{k-l}\sum_{k+l\leq l_{1}<\cdots<l_{u}\leq 2k-1}\frac{1}{l_{1}\cdots l_{u}}\\
  =&\theta \Big|(1+\frac{1}{2k-1})\cdots(1+\frac{1}{k+l})-1\Big|\\
 =&\theta \frac{k-l}{k+l}<\theta,
    \end{align*}
    and
    \begin{align*}
    \frac{\Gamma(k+l+\theta)}{2k\Gamma(2k+\theta)}=&\frac{1}{2k(2k+\theta-1)\cdots(k+l+\theta)}\\
    \leq&\frac{1}{2k(2k-1)\cdots(k+l)}=\frac{\Gamma(k+l)}{\Gamma(2k+1)}.
    \end{align*}
  Therefore, 
  \begin{align*}
  \Big|A_{k,p+1}(\theta)-A_{k,p+1}\Big|\leq&\theta p\sum_{l=p}^{k-1}\frac{2^{k}k!}{2^{l}l!}\frac{\Gamma(k+l)}{\Gamma(2k+1)}A_{l,p}+\theta\sum_{l=p}^{k-1}\frac{2^{k}k!}{2^{l}l!}\frac{\Gamma(k+l)}{\Gamma(2k+1)}A_{l,p}\\
  =&\theta(p+1)A_{k,p+1}.
    \end{align*}  
  Thus, we have proved the lemma.
  
  \end{proof}

  \textbf{[PROOF OF LEMMA \ref{FT}]:}
  
  \begin{proof}
  
  By Lemma \ref{UBMH}, we have
  \begin{align*}
 \sum_{l=[\lambda]+1}^{\infty}\theta^{l}\sum_{k=l}^{\infty}\frac{(\lambda\log\frac{1}{\theta})^{k}}{k!}A_{k,l}(\theta)
  \leq& \sum_{l=[\lambda]+1}^{\infty}\theta^{l}\sum_{k=l}^{\infty}\frac{(\lambda\log\frac{1}{\theta})^{k}}{k!}\frac{1}{2^{l-2}}\\
  =& 4\sum_{l=[\lambda]+1}^{\infty}\frac{\theta^{l}}{2^{l}}\sum_{k=l}^{\infty}\frac{(\lambda\log\frac{1}{\theta})^{k}}{k!} \\
  \leq &4\sum_{l=[\lambda]+1}^{\infty}(\frac{\theta}{2})^{l}e^{\lambda\log\frac{1}{\theta}}=4\sum_{l=[\lambda]+1}^{\infty}\frac{\theta^{l-\lambda}}{2^{l}}=4\frac{\theta^{[\lambda]-\lambda+1}}{2^{[\lambda]+1}}\frac{2}{2-\theta}\\
  \rightarrow&0,\mbox{ as } \theta\rightarrow0, \mbox{ due to }[\lambda]+1 >\lambda.
  \end{align*}
  Similarly, we can also show
  $$
  \lim_{\theta\rightarrow0}\sum_{l=[\lambda]+1}^{\infty}\theta^{l}\sum_{k=l}^{\infty}\frac{(\lambda\log\frac{1}{\theta})^{k}}{k!}A_{k+n,l}(\theta)=0.  
  $$
  \end{proof}
  
\textbf{ [PROOF OF LEMMA \ref{COEF1}]:}
\begin{proof} 
 By mathematical induction with respect to $p$, we can prove this lemma. For $p=1$, by Stirling's formula,
 \begin{align}
 \Gamma(z)\sim \sqrt{\frac{2\pi}{z}}(\frac{z}{e})^{z},\label{stirling_formula}
 \end{align}
 we have $A_{k,1}\sim\sqrt{\pi}\frac{1}{\sqrt{k}}(\frac{1}{2})^{k},$ as $k\rightarrow+\infty.$
 We can therefore assume that, as $k\rightarrow+\infty,$
  \begin{align}
 A_{k,p}\sim C_{p}\frac{1}{k^{\frac{p}{2}}}(\frac{p}{p+1})^{k}, \quad C_{p}=C_{p-1}\sqrt{\pi}(\frac{p+1}{p-1})^{\frac{p-1}{2}}.\label{coefficient_assumption}
 \end{align}
For $A_{k,p+1},$  we have
$$
A_{k,p+1}=\sum_{l=p}^{k-1}\frac{2^{k}k!\Gamma(k+l)}{2^{l}l!\Gamma(2k+1)}A_{l,p}.
$$ 
 By the assumption (\ref{coefficient_assumption}), $\forall \epsilon>0,\exists M>0,$ such that $\forall k>M$, 
 $$
 1-\epsilon<\frac{A_{k,p}}{C_{p}\frac{1}{k^{\frac{p}{2}}}(\frac{p}{p+1})^{k}}<1+\epsilon;
 $$
 then we rewrite $A_{k,p+1}$ as $X+Y$, where 
\begin{align*}
 &X=\sum_{l=p}^{M}\frac{2^{k}k!\Gamma(k+l)}{2^{l}l!\Gamma(2k+1)}A_{l,p},
&Y=\sum_{l=M+1}^{k-1}\frac{2^{k}k!\Gamma(k+l)}{2^{l}l!\Gamma(2k+1)}A_{l,p}.
 \end{align*}
  Define $a_{k}(l)=\frac{2^{k}k!\Gamma(k+l)}{2^{l}l!\Gamma(2k+1)}(\frac{p}{p+1})^{l}$, and $\Sigma_{1}=\sum_{l=p}^{k-1}\frac{C_{p}}{l^{\frac{p}{2}}}a_{k}(l).$ Now we are going to show 
  $
  \lim_{k\rightarrow+\infty}\frac{X}{\Sigma_{1}}=0,
  $
  and $\lim_{k\rightarrow+\infty}\frac{Y}{\Sigma_{1}}=1$.\\
   Indeed, since
$$
0\leq X\leq \frac{\max_{p\leq l\leq M}\{A_{l,p}\}(M-p+1)}{2^{p}p!}\frac{2^{k}k!\Gamma(k+M)}{\Gamma(2k+1)}
$$
and
$$
(1-\epsilon)\sum_{l=M+1}^{k-1}\frac{C_{p}}{l^{\frac{p}{2}}}a_{k}(l)\leq Y\leq (1+\epsilon)\sum_{l=M+1}^{k-1}\frac{C_{p}}{l^{\frac{p}{2}}}a_{k}(l),
$$  
 we have
  \begin{align*}
  0\leq &\frac{X}{\Sigma_{1}}\leq\frac{X}{\frac{C_{p}}{(k-1)^{\frac{p}{2}}}a_{k}(k-1)}
  \leq \frac{\max_{p\leq l\leq M}\{A_{l,p}\}(M-p+1)}{2^{p}p!}
  (k-1)^{\frac{p}{2}}\frac{2^{k}k!\Gamma(k+M)}{\Gamma(2k+1)a_{k}(k-1)}\\
  \leq&\frac{\max_{p\leq l\leq M}\{A_{l,p}\}(M-p+1)}{2^{p}p!}
  (k-1)^{\frac{p}{2}}  (\frac{2(p+1)}{p})^{k-1}\frac{\Gamma(k+M)\Gamma(k)}{\Gamma(2k-1)}.
  \end{align*}
  By Stirling's formula (\ref{stirling_formula}), we have, as $k\rightarrow+\infty$,
    \begin{align}
  (k-1)^{\frac{p}{2}}  (\frac{2(p+1)}{p})^{k-1}\frac{\Gamma(k+M)\Gamma(k)}{\Gamma(2k-1)}\sim k^{M+\frac{p+1}{2}}(\frac{p+1}{2p})^{k-1} \rightarrow0.\label{convergence}
   \end{align}
  Thus, $ \lim_{k\rightarrow+\infty}\frac{X}{\Sigma_{1}}=0.$ 
  Similarly, 
  $$
  0\leq\sum_{l=p}^{M}\frac{C_{p}}{l^{\frac{p}{2}}}a_{k}(l)\leq
 \frac{C_{p}(M-p+1)}{p^{\frac{p}{2}}2^{p}p!}(\frac{p}{p+1})^{p}\frac{2^{k}k!\Gamma(k+M)}{\Gamma(2k+1)},
    $$
  then 
  $$
  0\leq \frac{\sum_{l=p}^{M}\frac{C_{p}}{l^{\frac{p}{2}}}a_{k}(l)}{\Sigma_{1}}\leq 
  \frac{(M-p+1)(\frac{p}{p+1})^{p}}{p^{\frac{p}{2}}2^{p}p!}(k-1)^{\frac{p}{2}}(\frac{2(p+1)}{p})^{k-1}\frac{\Gamma(k+M)\Gamma(k)}{\Gamma(k-1)}.
  $$
  Thus, $\frac{\sum_{l=p}^{M}\frac{C_{p}}{l^{\frac{p}{2}}}a_{k}(l)}{\Sigma_{1}}\rightarrow0$ due to (\ref{convergence}) .
   Therefore, for the following inequality,
 $$
  (1-\epsilon)\left(1-\frac{\sum_{l=p}^{M}\frac{C_{p}}{l^{\frac{p}{2}}}a_{k}(l)}{\Sigma_{1}}\right)\leq \frac{Y}{\Sigma_{1}}\leq (1+\epsilon)\left(1-\frac{\sum_{l=p}^{M}\frac{C_{p}}{l^{\frac{p}{2}}}a_{k}(l)}{\Sigma_{1}}\right),
  $$
 if we let $k\rightarrow+\infty$, and $\epsilon\rightarrow0,$ then we have
  $$
  \lim_{k\rightarrow\infty} \frac{Y}{\Sigma_{1}}=1;
  $$
  hence $A_{k,p+1}\sim \Sigma_{1}$. Moreover,
  $$
  a_{k}(l)=\frac{2^{k}k!(k-1)!}{\Gamma(2k+1)}\left(\frac{2(p+1)}{p+2}\right)^{k}
  \frac{\Gamma(k+l)}{l!\Gamma(k)}\left(1-\frac{p+2}{2(p+1)}\right)^{l}\left(\frac{p+2}{2(p+1)}\right)^{k}.
  $$
  Let $X_{\alpha}^{k}$ be negative binomial, $NB(k,\alpha)$, where $\alpha=\frac{p+2}{2(p+1)}$, then
  $$
  a_{k}(l)=\frac{2^{k}k!(k-1)!}{\Gamma(2k+1)}\left(\frac{2(p+1)}{p+2}\right)^{k}P(X_{\alpha}^{k}=l),
  $$
  and 
  $$
  \Sigma_{1}=C_{p}\frac{2^{k}k!(k-1)!}{\Gamma(2k+1)}\left(\frac{2(p+1)}{p+2}\right)^{k}\sum^{k-1}_{l=p}\frac{1}{l^{\frac{p}{2}}}P(X_{\alpha}^{k}=l).
  $$
  We claim that, as  $k\rightarrow+\infty$,  
  \begin{align}
  \sum_{l=p}^{k-1}\frac{1}{l^{\frac{p}{2}}}P(X_{\alpha}^{k}=l)\sim \frac{1}{l_{0}^{\frac{p}{2}}}\sum_{l=p}^{k-1}P(X_{\alpha}^{k}=l)\sim \frac{1}{l_{0}^{\frac{p}{2}}}, \label{negative_binomial_claim}
  \end{align}
   where $l_{0}=\frac{(1-\alpha) k}{\alpha}$.
  Therefore, 
 $$
 \Sigma_{1}\sim \frac{1}{l_{0}^{\frac{p}{2}}}C_{p}\frac{4^{k}k!(k-1)!}{\Gamma(2k+1)}\left(\frac{p+1}{p+2}\right)^{k}.
 $$ 
 Then, by Stirling's formula (\ref{stirling_formula}), we know 
  $$
  A_{k,p+1}\sim \Sigma_{1}\sim C_{p}\sqrt{\pi}\left(\frac{p+2}{p}\right)^{p/2}\frac{1}{k^{\frac{p+1}{2}}}\left(\frac{p+1}{p+2}\right)^{k}.
  $$
 This lemma is thus proved. Now we only need to show claim (\ref{negative_binomial_claim}).
   Indeed,
  $$
  \frac{\sum_{l=p}^{k-1}\frac{1}{l^{\frac{p}{2}}}P(X_{\alpha}^{k}=l)}{\frac{1}{l_{0}^{\frac{p}{2}}}}=\sum_{l=p}^{k-1}\left(\sqrt{\frac{l_{0}}{l}}\right)^{p}P(X_{\alpha}^{k}=l).  
  $$
  $\forall\epsilon>0,$ we have
  \begin{align*}
&\sum_{l=p}^{k-1}\left(\sqrt{\frac{l_{0}}{l}}\right)^{p}P(X_{\alpha}^{k}=l)\\
=&
  \sum_{p\leq l\leq l_{0}(1-\epsilon)}\left(\sqrt{\frac{l_{0}}{l}}\right)^{p}P(X_{\alpha}^{k}=l)+\sum_{l_{0}(1-\epsilon) \leq l\leq l_{0}(1+\epsilon)}\left(\sqrt{\frac{l_{0}}{l}}\right)^{p}P(X_{\alpha}^{k}=l)\\
  +&\sum_{l_{0}(1+\epsilon) \leq l\leq k-1}\left(\sqrt{\frac{l_{0}}{l}}\right)^{p}P(X_{\alpha}^{k}=l), 
  \end{align*}
  where
  \begin{align*}
  0\leq \sum_{p\leq l\leq l_{0}(1-\epsilon)}\left(\sqrt{\frac{l_{0}}{l}}\right)^{p}P(X_{\alpha}^{k}=l)\leq &\left(\sqrt{\frac{l_{0}}{p}}\right)^{p}\sum_{p\leq l\leq l_{0}(1-\epsilon)}P(X_{\alpha}^{k}=l)\\
  =&\left(\sqrt{\frac{l_{0}}{p}}\right)^{p}P(X_{\alpha}^{k}\leq l_{0}(1-\epsilon)),
  \end{align*}
  \begin{align*}
   0\leq \sum_{ l_{0}(1+\epsilon)\leq l\leq k-1}\left(\sqrt{\frac{l_{0}}{l}}\right)^{p}P(X_{\alpha}^{k}=l)\leq &\left(\sqrt{\frac{l_{0}}{l_{0}(1+\epsilon)}}\right)^{p}\sum_{l_{0}(1+\epsilon)\leq l\leq k-1}P(X_{\alpha}^{k}=l)\\
   \leq&\left(\sqrt{\frac{1}{1+\epsilon}}\right)^{p}P(X_{\alpha}^{k}\geq l_{0}(1+\epsilon)) ,
 \end{align*}
  and
 \begin{align*}
 &\left(\sqrt{\frac{l_{0}}{l_{0}(1+\epsilon)}}\right)^{p}P(l_{0}(1-\epsilon)\leq X_{\alpha}^{k}
 \leq l_{0}(1+\epsilon))\\
 \leq&\sum_{l_{0}(1-\epsilon) \leq l\leq l_{0}(1+\epsilon)}\left(\sqrt{\frac{l_{0}}{l}}\right)^{p}P(X_{\alpha}^{k}=l)\\
 \leq&\left(\sqrt{\frac{l_{0}}{l_{0}(1-\epsilon)}}\right)^{p}P(l_{0}(1-\epsilon)\leq X_{\alpha}^{k}\leq l_{0}(1+\epsilon)) .
 \end{align*}
  By LDP for $NB(k,\alpha)$ in example \ref{LDP_EXAM}, we have
  $$
  P(X_{\alpha}^{k}\leq l_{0}(1-\epsilon))\sim e^{-k\left[\inf_{x<(-\epsilon +1)\frac{1-\alpha}{\alpha}} I_{1}(x)\right]},
  $$
  and
  $$
  P(X_{\alpha}^{k}\geq l_{0}(1+\epsilon))\sim e^{-k\left[\inf_{x>(\epsilon +1)\frac{1-\alpha}{\alpha}} I_{1}(x)\right]} .   
  $$
  Therefore,  as  $k\rightarrow+\infty$,
    $$
  \sum_{p\leq l\leq l_{0}(1-\epsilon)}\left(\sqrt{\frac{l_{0}}{l}}\right)^{p}P(X_{\alpha}^{k}=l)\rightarrow0 ,
  $$
  and 
  $$
  \sum_{l_{0}(1+\epsilon) \leq l\leq k-1}\left(\sqrt{\frac{l_{0}}{l}}\right)^{p}P(X_{\alpha}^{k}=l)  \rightarrow0. 
  $$
  By the central limit theorem of $X_{\alpha}^{k}$, if we let $k\rightarrow+\infty$, then $\epsilon\rightarrow0$, we have 
  $$
  \sum_{l_{0}(1-\epsilon) \leq l\leq l_{0}(1+\epsilon)}\left(\sqrt{\frac{l_{0}}{l}}\right)^{p}P(X_{\alpha}^{k}=l)  \rightarrow1. 
  $$
  The claim (\ref{negative_binomial_claim}) is thus proved.
  \end{proof}

  \textbf{ [PROOF OF LEMMA \ref{COEF2}]:}
  \begin{proof}
  Define 
  $$
  \Sigma_{2}=\sum_{s=0}^{k-l} \binom{k}{s}\left(\frac{\lambda-l}{\lambda}\right)^{s}C_{l}\frac{1}{(k-s)^{\frac{l}{2}}}\left(\frac{l}{l+1}\right)^{k-s}.
  $$
  Then 
  $$
  \Sigma_{2}=C_{l}\left(\frac{\lambda-l}{\lambda}+\frac{l}{l+1}\right)^{k}\sum_{s=0}^{k-l} \frac{1}{(k-s)^{\frac{l}{2}}}P(X_{\beta}^{k}=s),
  $$
  where $X_{\beta}^{k}$ follows binomial distribution $B(k,\beta)$, with $\beta=\frac{\frac{\lambda-l}{\lambda}}{\frac{\lambda-l}{\lambda}+\frac{l}{l+1}}.$
  
  Next, we show, as $k\rightarrow+\infty$,
    \begin{align}
  \sum_{s=0}^{k-l} \frac{1}{(k-s)^{\frac{l}{2}}}P(X_{\beta}^{k}=s)\sim\left( \frac{\beta}{(1-\beta)s_{0}}\right)^{\frac{l}{2}},\label{binomial_assumption}
  \end{align} 
  where $s_{0}=\beta k$. 
  To this end, $\forall \epsilon>0$, let us consider 
  $$
  s_{0}^{\frac{l}{2}}\sum_{s=0}^{k-l}\frac{1}{(k-s)^{\frac{l}{2}}}P(X_{\beta}^{k}=s)=\sum_{s=0}^{k-l}\left(\sqrt{\frac{s_{0}}{k-s}}\right)^{l}P(X_{\beta}^{k}=s).
  $$
  $\forall\epsilon>0,$ we have, 
 \begin{align*}
 &\sum_{s=0}^{k-l}\left(\sqrt{\frac{s_{0}}{k-s}}\right)^{l}P(X_{\beta}^{k}=s)\\
 =&\sum_{0\leq s\leq s_{0}(1-\epsilon)}\left(\sqrt{\frac{s_{0}}{k-s}}\right)^{l}P(X_{\beta}^{k}=s)+\sum_{(1-\epsilon)s_{0}\leq s\leq (1+\epsilon)s_{0}}\left(\sqrt{\frac{s_{0}}{k-s}}\right)^{l}P(X_{\beta}^{k}=s)\\
 +&\sum_{(1+\epsilon)s_{0}\leq s\leq k-1}\left(\sqrt{\frac{s_{0}}{k-s}}\right)^{l}P(X_{\beta}^{k}=s),
 \end{align*}
 where
  $$
  0\leq \sum_{0\leq s\leq s_{0}(1-\epsilon)}\left(\sqrt{\frac{s_{0}}{k-s}}\right)^{l}P(X_{\beta}^{k}=s)\leq  \left(\sqrt{\frac{s_{0}}{k-s_{0}(1-\epsilon)}}\right)^{l}P(X_{\beta}^{k}\leq s_{0}(1-\epsilon)) 
 $$
 and 
 $$
 0\leq \sum_{(1+\epsilon)s_{0}\leq s\leq k-1}\left(\sqrt{\frac{s_{0}}{k-s}}\right)^{l}P(X_{\beta}^{k}=s)\leq  \left(\sqrt{\frac{s_{0}}{1}}\right)^{l}P(X_{\beta}^{k}\geq s_{0}(1+\epsilon)).
 $$ 
 Then by the LDP for $B(k,\beta)$ in example \ref{LDP_EXAM}, we have 
 $$
 \lim_{k\rightarrow+\infty}\sum_{0\leq s\leq s_{0}(1-\epsilon)}\left(\sqrt{\frac{s_{0}}{k-s}}\right)^{l}P(X_{\beta}^{k}=s)=0,
 $$
and
$$
\lim_{k\rightarrow+\infty}\sum_{k-l\geq s\geq s_{0}(1+\epsilon)}\left(\sqrt{\frac{s_{0}}{k-s}}\right)^{l}P(X_{\beta}^{k}=s)=0.
$$
Moreover, 
\begin{align*}
&\left(\sqrt{\frac{s_{0}}{k-(1+\epsilon)s_{0}}}\right)^{l}P((1-\epsilon)s_{0}\leq X_{\beta}^{k}\leq (1+\epsilon)s_{0})\\
\leq &\sum_{(1-\epsilon)s_{0}\leq s\leq (1+\epsilon)s_{0}}\left(\sqrt{\frac{s_{0}}{k-s}}\right)^{l}P(X_{\beta}^{k}=s)\\
\leq&\left(\sqrt{\frac{s_{0}}{k-(1-\epsilon)s_{0}}}\right)^{l}P((1-\epsilon)s_{0}\leq X_{\beta}^{k}\leq (1+\epsilon)s_{0}).
\end{align*}
Analogously, if we let $k\rightarrow\infty$ and $\epsilon\rightarrow0$, then we have 
$$
\sum_{(1-\epsilon)s_{0}\leq s\leq (1+\epsilon)s_{0}}\left(\sqrt{\frac{s_{0}}{k-s}}\right)^{l}P(X_{\beta}^{k}=s)\rightarrow\left(\frac{\beta}{1-\beta}\right)^{\frac{l}{2}},
$$
due to the central limit theorem of binomial distributions. 
Therefore, 
$$
\Sigma_{2}\sim C_{l}\left(\frac{1}{1-\beta}\right)^{\frac{l}{2}}\frac{1}{k^{\frac{l}{2}}}\left(\frac{\lambda-l}{\lambda}+\frac{l}{l+1}\right)^{k}.
$$ 
 Next we are going to show $C_{k,l}\sim \Sigma_{2}.$ Because of Lemma \ref{COEF1} $, \forall \epsilon>0, \exists M> l$, such that $\forall k-s>M$,
$$
1-\epsilon\leq \frac{A_{k-s,l}}{C_{l}\frac{1}{(k-s)^{\frac{l}{2}}}(\frac{l}{l+1})^{k-s}}\leq 1+\epsilon. 
$$ 
Then we rewrite $C_{k,l}$ as $A+B,$
where 
$$
A=\sum_{s=0}^{k-M-1}\binom{k}{s}\left(\frac{\lambda-l}{\lambda}\right)^{s}A_{k-s,l}
$$
and
$$
B=\sum_{k-M\leq s\leq k-l}\binom{k}{s}\left(\frac{\lambda-l}{\lambda}\right)^{s}A_{k-s,l}.
$$
Since 
$$
0\leq B\leq \max_{k-M\leq s\leq k-l}\{A_{k-s,l}\}\binom{k}{k-M}\left(\frac{\lambda-l}{\lambda}\right)^{k-M},
$$
and 
\begin{align*}
0\leq \frac{B}{\Sigma_{2}}\leq&\frac{\max_{k-M\leq s\leq k-l}\{A_{k-s,l}\}\binom{k}{k-M}(\frac{\lambda-l}{\lambda})^{k-M}}{\Sigma_{2}}\\
\sim&\frac{\max_{k-M\leq s\leq k-l}\{A_{k-s,l}\}}{C_{l}M!}\left(\frac{\lambda}{\lambda-l}\right)^{M}(1-\beta)^{\frac{l}{2}}k^{\frac{l}{2}+M}\left(\frac{\frac{\lambda-l}{\lambda}}{\frac{\lambda-l}{\lambda}+\frac{l}{l+1}}\right)^{k}\\
\rightarrow&0.
\end{align*}
we can easily see $\lim_{k\rightarrow+\infty}\frac{B}{\Sigma_{2}}=0.$
Moreover,
\begin{align*}
&(1-\epsilon)\sum_{0\leq s\leq k-M-1}\binom{k}{s}\left(\frac{\lambda-l}{\lambda}\right)^{s}C_{l}\frac{1}{(k-s)^{\frac{l}{2}}}\left(\frac{l}{l+1}\right)^{k-s}\\
\leq &A\\
\leq &(1+\epsilon)\sum_{0\leq s\leq k-M-1}\binom{k}{s}\left(\frac{\lambda-l}{\lambda}\right)^{s}C_{l}\frac{1}{(k-s)^{\frac{l}{2}}}\left(\frac{l}{l+1}\right)^{k-s};
\end{align*}
and also
\begin{align*}
&(1-\epsilon)\left(\Sigma_{2}-\sum_{k-M\leq s\leq k-l}\binom{k}{s}\left(\frac{\lambda-l}{\lambda}\right)^{s}C_{l}\frac{1}{(k-s)^{\frac{l}{2}}}\left(\frac{l}{l+1}\right)^{k-s}\right)\\
\leq& A
\leq (1+\epsilon)\left(\Sigma_{2}-\sum_{k-M\leq s\leq k-l}\binom{k}{s}\left(\frac{\lambda-l}{\lambda}\right)^{s}C_{l}\frac{1}{(k-s)^{\frac{l}{2}}}\left(\frac{l}{l+1}\right)^{k-s}\right).
\end{align*}
Since
\begin{align*}
0\leq &\sum_{k-M\leq s\leq k-l}\binom{k}{s}\left(\frac{\lambda-l}{\lambda}\right)^{s}C_{l}\frac{1}{(k-s)^{\frac{l}{2}}}\left(\frac{l}{l+1}\right)^{k-s}\\
\leq&C_{l}\frac{1}{l^{\frac{l}{2}}}(\frac{l}{l+1})^{l}\binom{k}{k-M}\left(\frac{\lambda-l}{\lambda}\right)^{k-M},
\end{align*}
we have
\begin{align*}
0\leq&\frac{\sum_{k-M\leq s\leq k-l}\binom{k}{s}(\frac{\lambda-l}{\lambda})^{s}C_{l}\frac{1}{(k-s)^{\frac{l}{2}}}(\frac{l}{l+1})^{k-s}}{\Sigma_{2}}\\
\leq & \frac{C_{l}\frac{1}{l^{\frac{l}{2}}}(\frac{l}{l+1})^{l}\binom{k}{k-M}(\frac{\lambda-l}{\lambda})^{k-M}}{\Sigma_{2}}
\rightarrow0\mbox{ as } k\rightarrow\infty.
\end{align*}
Hence, letting $k\rightarrow\infty$, $\epsilon\rightarrow0$, we have $A\sim \Sigma_{2}$. Thus, $C_{k,l}\sim \Sigma_{2}$. Then,  
\begin{align*}
C_{k,l}\sim \Sigma_{2}\sim &C_{l}\left(\frac{1}{1-\beta}\right)^{\frac{l}{2}}\frac{1}{k^{\frac{l}{2}}}\left(\frac{\lambda-l}{\lambda}+\frac{l}{l+1}\right)^{k}\\
=&C_{l}\left(1+\frac{(\lambda-l)(l+1)}{\lambda l}\right)^{\frac{l}{2}}\frac{1}{k^{\frac{l}{2}}}\left(\frac{\lambda-l}{\lambda}+\frac{l}{l+1}\right)^{k}.
\end{align*}
\end{proof}
\textbf{[PROOF OF LEMMA \ref{COEF}]:}
\begin{proof}
Let us assume that 
\begin{align}
\lim_{\theta\rightarrow0}\tilde{K}_{n}^{\lambda}(\theta)=(\frac{u-1}{u})^{n} \mbox{ for } u(u-1)<\lambda\leq u(u+1),u>2; \label{final_assumption}
\end{align}
 then we are going to show $\lim_{\theta\rightarrow0}K_{n}^{\lambda}(\theta)=\lim_{\theta\rightarrow0}\tilde{K}_{n}^{\lambda}(\theta).$ Note that $K_{n}^{\lambda}(\theta)$ can be rewritten as 
$$
K_{n}^{\lambda}(\theta)=\frac{\tilde{K}_{n}^{\lambda}(\theta)+F_{n}^{\lambda}(\theta)}{1+G^{\lambda}(\theta)},
$$ 
where 
$$
F_{n}^{\lambda}(\theta)=\frac{\sum_{l=1}^{[\lambda]}\theta^{l}\sum_{k=l}^{\infty}\frac{(\lambda\log\frac{1}{\theta})^{k}}{k!}(A_{k+n,l}(\theta)-A_{k+n,l})}{\sum_{l=1}^{[\lambda]}\theta^{l}\sum_{k=l}^{\infty}\frac{(\lambda\log\frac{1}{\theta})^{k}}{k!}A_{k,l}},
$$
and
$$
G^{\lambda}(\theta)=\frac{\sum_{l=1}^{[\lambda]}\theta^{l}\sum_{k=l}^{\infty}\frac{(\lambda\log\frac{1}{\theta})^{k}}{k!}(A_{k,l}(\theta)-A_{k,l})}{\sum_{l=1}^{[\lambda]}\theta^{l}\sum_{k=l}^{\infty}\frac{(\lambda\log\frac{1}{\theta})^{k}}{k!}A_{k,l}}.
$$
We claim that $\lim_{\theta\rightarrow0}F_{n}^{\lambda}(\theta)=\lim_{\theta\rightarrow0}G^{\lambda}(\theta)=0.$
Then 
$$
\lim_{\theta\rightarrow0}K_{n}^{\lambda}(\theta)=\lim_{\theta\rightarrow0}\tilde{K}_{n}^{\lambda}(\theta)
$$ 
follows.
Indeed, by Lemma \ref{UBMH}, we have
\begin{align*}
0\leq|F_{n}^{\lambda}(\theta)|\leq &\frac{\sum_{l=1}^{[\lambda]}\theta^{l}\sum_{k=l}^{\infty}\frac{(\lambda\log\frac{1}{\theta})^{k}}{k!}|A_{k+n,l}(\theta)-A_{k+n,l}|}{\sum_{l=1}^{[\lambda]}\theta^{l}\sum_{k=l}^{\infty}\frac{(\lambda\log\frac{1}{\theta})^{k}}{k!}A_{k,l}}\\
\leq& \frac{\sum_{l=1}^{[\lambda]}\theta^{l}l\theta\sum_{k=l}^{\infty}\frac{(\lambda\log\frac{1}{\theta})^{k}}{k!}A_{k+n,l}}{\sum_{l=1}^{[\lambda]}\theta^{l}\sum_{k=l}^{\infty}\frac{(\lambda\log\frac{1}{\theta})^{k}}{k!}A_{k,l}}\\
\leq &[\lambda]\theta  \frac{\sum_{l=1}^{[\lambda]}\theta^{l}\sum_{k=l}^{\infty}\frac{(\lambda\log\frac{1}{\theta})^{k}}{k!}A_{k+n,l}}{\sum_{l=1}^{[\lambda]}\theta^{l}\sum_{k=l}^{\infty}\frac{(\lambda\log\frac{1}{\theta})^{k}}{k!}A_{k,l}}=[\lambda]\theta \tilde{K}_{n}^{\lambda}(\theta)\rightarrow0.
\end{align*}
Similarly, by Lemma \ref{UBMH}, we also have
\begin{align*}
0\leq |G^{\lambda}(\theta)|\leq [\lambda]\theta\rightarrow0,\mbox{ as } \theta\rightarrow0. 
\end{align*}
Now we are going to show the assumption (\ref{final_assumption}). 
We can rewrite $\tilde{K}_{n}(\theta)$ as
$$
\sum_{v=1}^{[\lambda]}\frac{\theta^{v}\sum_{k=v}^{\infty}\frac{(\lambda\log\frac{1}{\theta})^{k}}{k!}A_{k,v}}{\sum_{l=1}^{[\lambda]}\theta^{l}\sum_{k=l}^{\infty}\frac{(\lambda\log\frac{1}{\theta})^{k}}{k!}A_{k,l}}\frac{\sum_{k=v}^{\infty}\frac{(\lambda\log\frac{1}{\theta})^{k}}{k!}A_{k+n,v}}{\sum_{k=v}^{\infty}\frac{(\lambda\log\frac{1}{\theta})^{k}}{k!}A_{k,v}}.
$$
Then,
$$
\lim_{\theta\rightarrow0}\tilde{K}_{n}^{\lambda}(\theta)=\sum_{v=1}^{[\lambda]}\lim_{\theta\rightarrow0}\frac{\theta^{v}\sum_{k=v}^{\infty}\frac{(\lambda\log\frac{1}{\theta})^{k}}{k!}A_{k,v}}{\sum_{l=1}^{[\lambda]}\theta^{l}\sum_{k=l}^{\infty}\frac{(\lambda\log\frac{1}{\theta})^{k}}{k!}A_{k,l}}\lim_{\theta\rightarrow0}\frac{\sum_{k=v}^{\infty}\frac{(\lambda\log\frac{1}{\theta})^{k}}{k!}A_{k+n,v}}{\sum_{k=v}^{\infty}\frac{(\lambda\log\frac{1}{\theta})^{k}}{k!}A_{k,v}}.
$$
By Lemma \ref{ML} and Lemma \ref{COEF1}, we know
$$
\lim_{\theta\rightarrow0}\frac{\sum_{k=v}^{\infty}\frac{(\lambda\log\frac{1}{\theta})^{k}}{k!}A_{k+n,v}}{\sum_{k=v}^{\infty}\frac{(\lambda\log\frac{1}{\theta})^{k}}{k!}A_{k,v}}=(\frac{v}{v+1})^{n}.
$$
Then we need to show 
\begin{align}
\lim_{\theta\rightarrow0}\frac{\theta^{v}\sum_{k=v}^{\infty}\frac{(\lambda\log\frac{1}{\theta})^{k}}{k!}A_{k,v}}{\sum_{l=1}^{[\lambda]}\theta^{l}\sum_{k=l}^{\infty}\frac{(\lambda\log\frac{1}{\theta})^{k}}{k!}A_{k,l}}=\delta_{(u-1)}(v).\label{limit_coef}
\end{align}
Once we have obtained this, then $\lim_{\theta\rightarrow0}\tilde{K}_{n}^{\lambda}(\theta)=\sum_{v=1}^{[\lambda]}\delta_{(u-1)}(v)(\frac{v}{v+1})^{n}=(\frac{u-1}{u})^{n}.$ To this end, both the numerator and the denumerator of 
$$
\frac{\theta^{v}\sum_{k=v}^{\infty}\frac{(\lambda\log\frac{1}{\theta})^{k}}{k!}A_{k,v}}{\sum_{l=1}^{[\lambda]}\theta^{l}\sum_{k=l}^{\infty}\frac{(\lambda\log\frac{1}{\theta})^{k}}{k!}A_{k,l}},
$$
are divided by $\theta^{\lambda}$.
 Thus, we need consider
$$
\frac{\theta^{v-\lambda}\sum_{k=v}^{\infty}\frac{(\lambda\log\frac{1}{\theta})^{k}}{k!}A_{k,v}}{\sum_{l=1}^{[\lambda]}\theta^{l-\lambda}\sum_{k=l}^{\infty}\frac{(\lambda\log\frac{1}{\theta})^{k}}{k!}A_{k,l}}.
$$
Since $1\leq v\leq [\lambda]\leq \lambda$, it is not difficult to see that 
\begin{align*}
\theta^{v-\lambda}\sum_{k=v}^{\infty}\frac{(\lambda\log\frac{1}{\theta})^{k}}{k!}A_{k,v}=&(\frac{1}{\theta})^{\lambda-v}\sum_{k=v}^{\infty}\frac{(\lambda\log\frac{1}{\theta})^{k}}{k!}A_{k,v}\\
 =&\left(\sum_{s=0}^{\infty}\frac{(\log\frac{1}{\theta})^{s}}{s!}(\lambda-v)^{s}\right)\left(\sum_{k=v}^{\infty}\frac{(\lambda\log\frac{1}{\theta})^{k}}{k!}A_{k,v}\right)\\
 =&\sum_{k=v}^{\infty}\frac{(\lambda\log\frac{1}{\theta})^{k}}{k!}C_{k,v},
\end{align*}
where 
$$
C_{k,v}=\sum_{s=0}^{k-v}\binom{k}{s}\left(\frac{\lambda-v}{\lambda}\right)^{s}A_{k-s,v}.
$$
Then 
$$
\frac{\theta^{v-\lambda}\sum_{k=v}^{\infty}\frac{(\lambda\log\frac{1}{\theta})^{k}}{k!}A_{k,v}}{\sum_{l=1}^{[\lambda]}\theta^{l-\lambda}\sum_{k=l}^{\infty}\frac{(\lambda\log\frac{1}{\theta})^{k}}{k!}A_{k,l}}=\frac{\sum_{k=v}^{\infty}\frac{(\lambda\log\frac{1}{\theta})^{k}}{k!}C_{k,v}}{\sum_{l=1}^{[\lambda]}\sum_{k=l}^{\infty}\frac{(\lambda\log\frac{1}{\theta})^{k}}{k!}C_{k,l}}.
$$
Thus, to figure out the limit (\ref{limit_coef}), we must find the leading term among 
$$
\sum_{k=l}^{\infty}\frac{(\lambda\log\frac{1}{\theta})^{k}}{k!}C_{k,l}, 1\leq l\leq [\lambda].
$$ 
By Lemma \ref{ML} and Lemma \ref{COEF2}, we have
\begin{align*}
&\lim_{\theta\rightarrow0}\frac{\sum_{k=v}^{\infty}\frac{(\lambda\log\frac{1}{\theta})^{k}}{k!}C_{k,v}}{\sum_{k=l}^{\infty}\frac{(\lambda\log\frac{1}{\theta})^{k}}{k!}C_{k,l}}=\lim_{k\rightarrow+\infty}\frac{C_{k,v}}{C_{k,l}}\\
=&\lim_{k\rightarrow+\infty}\frac{(1+\frac{(\lambda-v)(v+1)}{\lambda v})^{\frac{v}{2}}}{(1+\frac{(\lambda-l)(l+1)}{\lambda l})^{\frac{l}{2}}}\frac{1}{k^{\frac{v-l}{2}}}\left(\frac{\frac{\lambda-v}{\lambda}+\frac{v}{v+1}}{\frac{\lambda-l}{\lambda}+\frac{l}{l+1}}\right)^{k}.
\end{align*}
To find the leading term, we need to figure out the maximum term among $\frac{\lambda-l}{\lambda}+\frac{l}{l+1}, 1\leq l\leq [\lambda].$ Consider $f(x)=\frac{\lambda-x}{\lambda}+\frac{x}{x+1}=2-(\frac{x}{\lambda}+\frac{1}{x+1});$ then
$$
f^{'}(x)=\frac{1}{(x+1)^{2}}-\frac{1}{\lambda}.
$$
We know $f^{'}(x)$ is
$$
\begin{cases}
\geq0, & \mbox{ if }x\leq \sqrt{\lambda}-1\\
< 0,  &\mbox{ if } x> \sqrt{\lambda}-1.
\end{cases}
$$
Therefore, $\frac{\lambda-l}{\lambda}+\frac{l}{l+1}$ attains its maximum at $[\sqrt{\lambda}]-1$ or $[\sqrt{\lambda}]$.

\textbf{Case 1:} For $(u-1)u<\lambda<u^{2}, [\sqrt{\lambda}]=u-1$, since
$$
f(u-2)=2-(\frac{u-2}{\lambda}+\frac{1}{u-1})
$$
and
$$
f(u-1)=2-(\frac{u-1}{\lambda}+\frac{1}{u}),
$$
 $f(u-2)-f(u-1)=\frac{1}{\lambda}-\frac{1}{u(u-1)}<0.$ So $f(u-1)$ is the maximum term.
$$
\frac{\frac{\lambda-l}{\lambda}+\frac{l}{l+1}}{\frac{\lambda-u+1}{\lambda}+\frac{u-1}{u}}<1,\forall 1\leq l\leq[\lambda], l\neq u-1.
$$
Thus
$$
\lim_{k\rightarrow+\infty}\frac{C_{k,l}}{C_{k,u-1}}=\delta_{u-1}(l),\forall 1\leq l\leq[\lambda].
$$

\textbf{Case 2:} For $u^{2}\leq \lambda<u(u+1), [\sqrt{\lambda}]=u$; then the maximum term of $\frac{\lambda-l}{\lambda}+\frac{l}{l+1}$ should be $f(u-1)$ or $f(u)$. Since
$$
f(u-1)-f(u)=\frac{1}{\lambda}-\frac{1}{u(u+1)}>0,
$$
we can see $f(u-1)$ is the maximum term; and 
$$
\frac{\frac{\lambda-l}{\lambda}+\frac{l}{l+1}}{\frac{\lambda-u+1}{\lambda}+\frac{u-1}{u}}<1,\forall 1\leq l\leq[\lambda], l\neq u-1.
$$
Thus, 
$$
\lim_{k\rightarrow+\infty}\frac{C_{k,l}}{C_{k,u-1}}=\delta_{u-1}(l),\forall 1\leq l\leq[\lambda].
$$

\textbf{Case 3:} For $\lambda=u(u+1), [\sqrt{\lambda}]=u$, then the maximum term of $\frac{\lambda-l}{\lambda}+\frac{l}{l+1}$ is $f(u-1)$ or $f(u)$. Since
$$
f(u-1)-f(u)=\frac{1}{\lambda}-\frac{1}{u(u+1)}=0,
$$
one can have 
$$
\frac{\frac{\lambda-l}{\lambda}+\frac{l}{l+1}}{\frac{\lambda-u+1}{\lambda}+\frac{u-1}{u}}<1,\forall 1\leq l\leq[\lambda], l\neq u,u-1,
$$
and
$$
\frac{\frac{\lambda-u+1}{\lambda}+\frac{u-1}{u}}{\frac{\lambda-u}{\lambda}+\frac{u}{u+1}}=1.
$$
But 
$$
\lim_{k\rightarrow+\infty}\frac{C_{k,u}}{C_{k,u-1}}=\lim_{k\rightarrow+\infty}\frac{(1+\frac{(\lambda-u)(u+1)}{\lambda u})^{\frac{u}{2}}}{(1+\frac{(\lambda-u+1)u}{\lambda(u-1)})^{\frac{u-1}{2}}}\frac{1}{k^{\frac{1}{2}}}=0;
$$
therefore, $C_{k,u-1}$ is the leading term among $C_{k,l}, 1\leq l\leq [\lambda]$. We have
$$
\lim_{k\rightarrow+\infty}\frac{C_{k,l}}{C_{k,u-1}}=\delta_{u-1}(l),1\leq l\leq [\lambda].
$$
Thus, 
$$
\lim_{\theta\rightarrow0}\frac{\theta^{v}\sum_{k=v}^{\infty}\frac{(\lambda\log\frac{1}{\theta})^{k}}{k!}A_{k,v}}{\sum_{l=1}^{[\lambda]}\theta^{l}\sum_{k=l}^{\infty}\frac{(\lambda\log\frac{1}{\theta})^{k}}{k!}A_{k,l}}=\delta_{(u-1)}(v), 1\leq v\leq [\lambda].
$$
\end{proof}
\section{Acknowledgement}  
 This work is part of my PhD thesis. I would like to thank my PhD supervisor Shui Feng for suggesting this problem to me. Without him, I probably would never have the chance to obtain this work. I also would like to thank anonymous referees for insightful comments.
  
\bibliography{final}{}

\def\cprime{$'$} \def\cprime{$'$} \def\cprime{$'$} \def\cprime{$'$}
  \def\cprime{$'$}
\begin{thebibliography}{}

\bibitem[Dawson and Feng, 2006]{MR2244425}
Dawson, D.~A. and Feng, S. (2006).
\newblock Asymptotic behavior of the {P}oisson-{D}irichlet distribution for
  large mutation rate.
\newblock {\em Ann. Appl. Probab.}, 16(2):562--582.

\bibitem[Dembo and Zeitouni, 2010]{MR2571413}
Dembo, A. and Zeitouni, O. (2010).
\newblock {\em Large deviations techniques and applications}, volume~38 of {\em
  Stochastic Modelling and Applied Probability}.
\newblock Springer-Verlag, Berlin.
\newblock Corrected reprint of the second (1998) edition.

\bibitem[Ethier and Kurtz, 1981]{MR615945}
Ethier, S.~N. and Kurtz, T.~G. (1981).
\newblock The infinitely-many-neutral-alleles diffusion model.
\newblock {\em Adv. in Appl. Probab.}, 13(3):429--452.

\bibitem[Ethier and Kurtz, 1998]{MR1626158}
Ethier, S.~N. and Kurtz, T.~G. (1998).
\newblock Coupling and ergodic theorems for {F}leming-{V}iot processes.
\newblock {\em Ann. Probab.}, 26(2):533--561.

\bibitem[Feng, 2009]{MR2519357}
Feng, S. (2009).
\newblock Poisson-{D}irichlet distribution with small mutation rate.
\newblock {\em Stochastic Process. Appl.}, 119(6):2082--2094.

\bibitem[Gillespie, 1999]{Gillespie_role_of_size}
Gillespie, J.~H. (1999).
\newblock The role of population size in molecular evolution.
\newblock {\em Theoretical Population Biology}, 55(2):145--156.

\bibitem[Joyce et~al., 2003]{MR1951997}
Joyce, P., Krone, S.~M., and Kurtz, T.~G. (2003).
\newblock When can one detect overdominant selection in the infinite-alleles
  model?
\newblock {\em Ann. Appl. Probab.}, 13(1):181--212.

\bibitem[Kingman et~al., 1975]{MR0368264}
Kingman, J. F.~C., Taylor, S.~J., Hawkes, A.~G., Walker, A.~M., Cox, D.~R.,
  Smith, A. F.~M., Hill, B.~M., Burville, P.~J., and Leonard, T. (1975).
\newblock Random discrete distribution.
\newblock {\em J. Roy. Statist. Soc. Ser. B}, 37:1--22.
\newblock With a discussion by S. J. Taylor, A. G. Hawkes, A. M. Walker, D. R.
  Cox, A. F. M. Smith, B. M. Hill, P. J. Burville, T. Leonard and a reply by
  the author.

\bibitem[Maruyama and Nei, 1981]{MR64600}
Maruyama, T. and Nei, M. (1981).
\newblock Genetic variability maintained by mutation and overdominant selection
  in finite population.
\newblock {\em Genetics}, 98(2):441--459.

\bibitem[Zhou, 2010]{MThesis}
Zhou, Y. (2010).
\newblock The limits of certain probability distributions associated with the
  wright-fisher model.
\newblock {\em Open Access Dissertations and Thesis}, Paper 4430.

\end{thebibliography}
\bibliographystyle{apalike} 

\end{document}